\documentclass[11pt,leqno]{article}
\usepackage[sc]{mathpazo}
\usepackage[OT2,T1]{fontenc}
\usepackage[utf8]{inputenc}
\usepackage[top = 3cm, bottom  = 3cm, left = 3.7cm, right = 3.2cm]{geometry}
\usepackage{amsmath, amscd, amssymb, amsthm, amsfonts, euscript, mathtools, tikz-cd}
\usepackage{adjustbox}
\usepackage{url}
\usepackage{hyperref}
\hypersetup{ colorlinks=true, citecolor= black,linkcolor = black}
\usepackage{enumitem}

\usepackage{todonotes}
{\qedhere}

\newcommand{\Oo}{\mathcal{O}}
\DeclareSymbolFont{cyrletters}{OT2}{wncyr}{m}{n}
\DeclareMathSymbol{\Sha}{\mathalpha}{cyrletters}{"58}
\DeclareMathSymbol{\Bcyr}{\mathalpha}{cyrletters}{"42}

\DeclareMathOperator{\Spec}{\mathrm{Spec}}
\newcommand{\lau}[1]{(\!(#1)\!)}

\theoremstyle{definition}
\newtheorem{thm}{Theorem}[section]
\newtheorem*{thm*}{Theorem}
\newtheorem{proposition}[thm]{Proposition}
\newtheorem{defi}[thm]{Definition}
\newtheorem{fact}[thm]{Fact}
\newtheorem*{fact*}{Fact}
\newtheorem*{defi*}{Definition}
\newtheorem{rem}[thm]{Remark}
\newtheorem{example}[thm]{Example}
\newtheorem{lemma}[thm]{Lemma}
\newtheorem{cor}[thm]{Corollary}
\newtheorem*{cor*}{Corollary}
\newtheorem{rmk}[thm]{Remark}

\newtheorem{thmalpha}{Theorem}

\newcommand{\bp}{\begin{proposition}}
\newcommand{\nc}{\newcommand}
\nc{\N}{\mathbf{N}}
\nc{\Q}{\mathbf{Q}}
\nc{\Z}{\mathbf{Z}}
\nc{\F}{\mathbf{F}}
\nc{\Xx}{\mathcal{X}}

\title{\vspace*{-2ex}Transfer principles 
and the Kato-Kuzumaki conjecture}
\author{Felipe Gambardella \\ \vspace*{-1ex} \small \textit{École polytechnique de Paris (CMLS)} \and  Konstantinos Kartas \\ \vspace*{-1ex} \small \textit{University of Münster}}
\date{}
\begin{document}

\maketitle

\begingroup
\renewcommand\thefootnote{}

\endgroup
\begin{abstract}
\noindent We show that for tame valued fields of equal characteristic with divisible value group, the $C_i$ property  lifts from the residue field to the valued field under suitable hypotheses on the residue field. We apply this transfer principle to prove Kato–Kuzumaki's conjecture in full generality for several arithmetically significant fields, for instance the field $\mathbf{C}(x_1,\cdots,x_m)\lau{t_1}\cdots\lau{t_n}$, and the perfections of both $\overline{\mathbf{F}}_p(x_1,\cdots,x_m)\lau{t_1}\cdots\lau{t_n}$ and $\mathbf{F}_p\lau{t_1}\cdots\lau{t_n}$. 
Finally, we prove that $\mathbf{Q}_p$ satisfies the strong $C_1^1$ property, thereby answering a question of Wittenberg.
\end{abstract}
\tableofcontents
\section*{Introduction}

There are several notions of ``dimension'' of a field $k$.  One example is the \emph{Diophantine dimension}: for $i\in \mathbf{N}$, we say that $k$ is $C_i$ if every non-constant homogeneous polynomial $f(X_0,\dots,X_n)\in k[X_0,\dots,X_n]$ of degree $d$ with $d^i\leq n$ has a non-trivial zero over $k$. A $C_i$ field is clearly $C_j$ for every $j\geq i$. The least non-negative integer $i$ for which $k$ is $C_i$ is called its \emph{Diophantine dimension}. Thus, knowing the Diophantine dimension of $k$ gives us some information about the homogeneous polynomial equations which admit non-trivial solutions in $k$. \par

Another relevant notion of dimension is the \emph{cohomological dimension} of $k$. The precise relationship between Diophantine dimension and cohomological dimension is not fully understood. For instance, the question of whether a $C_i$ field has cohomological dimension at most $i$ is a prominent open question raised by Serre. This question is known to have a positive answer for $i\leq 2$; the case $i=1$ being rather elementary, while the case $i=2$ relies on deep results in Galois cohomology due to Suslin, see \cite[II~\S~4.5]{Serre1965CohomologieGaloisienne}. On the other hand, we can find examples of fields with cohomological dimension $1$ which are not $C_1$ —not even $C_i$ for any $i$— in \cite{ax}, showing that the converse implication fails in general. Moreover, any $p$-adic field gives a counter-example to this converse because it has cohomological dimension $2$ and it is not $C_i$ for any $i$, see \cite{Alemu1985QpnotCi} and \cite{ArkhipovKaratsuba1982QpnotCi}. \par

Looking for a Diophantine characterisation of the cohomological dimension, Kato and Kuzumaki introduced the $C_i^q$ property in \cite{KK1986DimensionOfFields}, which is a weakening of the $C_i$ property formulated in terms of Milnor $K$-theory of degree $q$. Their main conjecture asserts that for every pair of non-negative integers $i$ and $q$ a field satisfies $C_i^q$ if and only if its cohomological dimension is at most $i+q$. This conjecture is known to hold when $i=0$ as a consequence of the Bloch-Kato conjecture (now the norm-residue isomorphism theorem due to the work of Suslin, Joukhovitsky, Rost, Voevodsky, among others). Unfortunately, the conjecture is false in general. For instance, an example of a field of cohomological dimension $2$ which is not $C_2^0$ can be found in \cite[Proposition~2]{Merkurjev1991NotC20}. Moreover, a construction of a field of cohomological dimension $1$ that is not $C_1^0$ can be found in \cite{CTMadore2004NotC10}. However, these counter-examples are very far from the fields that one usually encounters in arithmetic geometry. It is therefore reasonable to consider Kato–Kuzumaki's conjecture open for the fields that commonly occur in arithmetic geometry and number theory.\par

Let us now recall the Kato-Kuzumaki's properties. Let $k$ be a field. For any finite extension $k'/k$, Milnor $\mathrm{K}$-theory admits a norm map $\mathrm{N}_{k'/k}: \mathrm{K}_q(k') \to \mathrm{K}_q(k)$, see \cite[Section~7.3]{GilleSzamuely2017CentralGalois}. For a $k$-variety $Z$, the norm group $\mathrm{N}_q(Z/k)$ is defined as the subgroup of $\mathrm{K}_q(k)$ generated by $\mathrm{N}_{k'/k}(\mathrm{K}_q(k'))$, where $k'/k$ ranges over all finite extensions such that $Z(k')~\neq~\emptyset$. For a pair of non-negative integers $i,q$, the $C_i^q$ property is defined as follows: for every pair of positive integers $d, n$ satisfying $d^i \leq n$, finite extension $k'/k$, and hypersurface $Z \subseteq \mathbf{P}^n_{k'}$ of degree $d$ we have $\mathrm{K}_q(k') = \mathrm{N}_q(Z/k')$. For example, a field $k$ satisfies $C_i^0$ if and only if for every pair of positive integers $d,n$ such that $d^i \leq n$ and field extension $k'/k$ any hypersurface $Z$ in $\mathbf{P}^n_{k'}$ of degree $d$ admits a $0$-cycle of degree $1$. In the other extreme, a field $k$ satisfies $C_0^q$ if and only if for every tower $k''/k'/k$ of finite extensions the norm map $\mathrm{N}_{k''/k'}: \mathrm{K}_{q}(k'') \to \mathrm{K}_q(k')$ is surjective. \par 

Verifying the conjecture or even the $C_i^q$ property for particular values of $i$ and $q$ over specific fields has proven to be quite challenging. In this article we are particularly interested in the case of henselian valued fields. To the knowledge of the authors, the only known cases of non-$C_1$ henselian valued fields satisfying Kato-Kuzumaki's conjecture are $\mathbf{C}\lau{t_1}\lau{t_2}$ and  $\mathbf{C}(x_1,\cdots, x_m)\lau{t}$ corresponding to \cite[Corollaire~4.7]{Wittenberg2015KKQp} and \cite[Theorem~3.9]{Diego2018KK} respectively. However, some particular properties are known for other fields, e.g. in \cite[Corollaire~5.5 and Corollaire~4.7]{Wittenberg2015KKQp} it is proved that $p$-adic fields satisfy $C_1^1$, the field $\mathbf{C}\lau{t_1}\dots\lau{t_n}$ satisfies $C_1^{n-1}$, and both $\mathbf{Q}_p\lau{t_2}\dots\lau{t_n}$ and $\mathbf{F}_p\lau{t_1}\dots \lau{t_n}$ satisfy $C_1^n$ \textit{away from} $p$. \par

In this article, we substantially generalize all of the aforementioned results. Notably, we establish Kato-Kuzumaki's conjecture for the field $\mathbf{C}(x_1,\cdots,x_m)\lau{t_1}\cdots\lau{t_n}$ and the perfections of $\overline{\mathbf{F}}_p(x_1,\cdots,x_m)\lau{t_1}\cdots\lau{t_n}$ and $\mathbf{F}_p\lau{t_1}\cdots\lau{t_n}$. We also prove that $\mathbf{Q}_p$ is strongly $C_1^1$, thereby answering a question of Wittenberg. Our approach is also novel in that it utilizes techniques from the model theory of valued fields.

\subsection*{Main results}
In the first part of this article, we introduce a strong version of the $C_i$ property, asserting that it persists along suitable henselian valuations.
\begin{defi*}
    Let $k$ be a perfect field and $i\in \mathbf{N}$. We say that $k$ is stably $C_i$ if every equicharacteristic tame field $(K,v)$ with divisible value group and residue field $k$ is $C_i$.
\end{defi*}
\noindent Note that every stably $C_i$ field is $C_i$, since the residue field of a $C_i$ valued field is itself $C_i$. We refer the reader to Section~\ref{akesec} for the notion of tameness. In equicharacteristic zero, tameness is equivalent to henselianity, whereas in mixed and positive characteristic it is a genuinely stronger condition. When combined with the divisibility of the value group, it equivalently says that $K$ is perfect and the natural map $G_K\to G_k$ is an isomorphism. In particular, we see that stably $C_0$ fields are precisely the algebraically closed fields. In general, we do not know if all $C_i$ fields are stably $C_i$, but the next result allows us to prove that this is the case for many classical $C_i$ fields.  
\begin{thmalpha}\label{thmalpha transfer intro}
Let $k$ be a perfect field. Then:
\begin{enumerate}
    \item If $k$ is stably $C_i$ and $l/k$ is a perfect field extension of transcendence degree $j$, then $l$ is stably $C_{i+j}$.
 \label{3} \item If $k$ is $C_i$ and is elementarily equivalent (in the language of rings) to a tame field, then $k$ is stably $C_i$.
\end{enumerate}
\end{thmalpha}

\noindent 
Part~(1) may be viewed as a stable analogue of the Lang--Nagata transition theorems for the $C_i$ property. 
We emphasize that part~(2) applies to fields that are elementarily equivalent to a tame field but do not necessarily admit a valuation making them tame.  We refer to such fields  as \textit{$t$-tame}. This flexibility is useful for applications---essentially all perfect henselian fields one encounters are $t$-tame, though not necessarily tame. This phenomenon was first observed in~\cite{JK}; for the reader’s convenience, we isolate the case needed here in Proposition~\ref{prop perfect microbial is t-tame}.

Theorem \ref{thmalpha transfer intro} allow us to prove that many classical $C_i$ fields are stably $C_i$. For example, since algebraically closed fields are stably $C_0$, part (1) implies that any perfected function field of transcendence degree $i$  over an algebraically closed field is stably $C_i$. Another example, used throughout, is that if $k$ is $C_i$, then the perfection of $k(\!(t)\!)$ is stably $C_{i+1}$. This is because  $k\lau t^{\mathrm{perf}}$ is $C_{i+1}$ by Greenberg-type arguments and is $t$-tame---though crucially not tame---so part (2) implies that it is stably $C_{i+1}$.

The second part of the article applies Theorem \ref{thmalpha transfer intro} to the Kato–Kuzumaki conjecture for iterated Laurent series. Our applications follow from a general transfer principle, in the spirit of \cite[Théorème~4.2]{Wittenberg2015KKQp}, whose statement we now recall. Wittenberg introduced the strong $C_1^q$ property (see Definition~\ref{def strong C1q}) in order to obtain the following transfer result (the $C_0^{q+1}$ hypothesis in \cite[Théorème~4.1]{Wittenberg2015KKQp} turns out to be superfluous, as was pointed out in \cite[Remark~3.3]{DiegoLuco2022HomogeneousKtheory}).
\begin{thm*}{\cite[Théorème~4.2]{Wittenberg2015KKQp}}
    Let $k$ be a field and $R$ an excellent henselian discrete valuation ring with $k$ as residue field. Let $K$ be the fraction field of $R$ and $\ell$ a prime number invertible in $R$. If $k$ satisfies the strong $C_1^q$ property at $\ell$, then $K$ satisfies the strong $C_1^{q+1}$ property at $\ell$.
\end{thm*}
\noindent We introduce a strengthening of the $C_i^q$ property for any non-negative integer $i$, called the stably uniform $C_i^q$ property (Definition~\ref{def stably uniform Ciq property}). This property asks for the existence of stably $C_i$ algebraic extensions of the field realising a generating set of Milnor $\mathrm{K}$-theory as norms of these extensions. The stably uniform $C_i^q$ property turns out to be amenable to a transfer principle.
\begin{thmalpha}\label{thmalpha KK iterated Laurent}
    Let $k$ be a perfect field and $i,q$ a pair of positive integers. Assume that $k$ satisfies the stably uniform $C_{i-1}^q$ and $C_i^{q-1}$ property. Let $R_0$ be an equicharacteristic henselian discrete valued field with residue field $k$ and $K$ the perfection of the fraction field of $R_0$. Then, the field $K$ satisfies the stably uniform $C_i^q$ property.
\end{thmalpha}
\noindent A similar transfer principle is expected for the usual $C_i^q$ properties, but  is currently only known when one restricts to hypersurfaces of prime degree different from the characteristic, see \cite[Theorem~1(3)]{KK1986DimensionOfFields}. Moreover, the expected transfer result does not include taking the perfection. It is needed for our result, as the model-theoretic tools we use are currently available only for perfect henselian valued fields.
\par
Theorem \ref{thmalpha KK iterated Laurent} admits many interesting~corollaries.
\begin{cor*}
    Let $k_0$ be an algebraically closed field. Then, the perfection of the field
    $K:=k_0(x_1,\dots, x_n) \lau{t_1}\dots\lau{t_m}$
    satisfies the stably uniform $C_i^q$ property (hence the $C_i^q$ property)
    for every pair of non-negative integers with $i+q \geq n+m$.
\end{cor*}
\noindent As mentioned earlier, this was previously known only when $K$ is either $\mathbf{C}(x_1,\ldots,x_m)$, $\mathbf{C}\lau{t_1}\lau{t_2}$, or $\mathbf{C}(x_1,\ldots,x_m)\lau{t}$.
\begin{cor*}
    The perfection of $\mathbf{F}_p\lau{t_1}\dots\lau{t_n}$ satisfies the stably uniform $C_i^q$ property (hence the $C_i^q$ property) for every pair of non-negative integers such that $i+q \geq n+1$.
\end{cor*}
\noindent Moreover, there is no need to take the perfection when $i=1$ and that property also extend to $p$-adic fields. That is, we have the following result.
\begin{cor*}
The fields $\mathbf{F}_p\lau{t_1}\dots \lau{t_n}$ and $\mathbf{Q}_p\lau{t_1}\dots\lau{t_{n-1}}$ satisfy the $C_1^n$ property.    
\end{cor*}
\noindent This improves \cite[Corollaire~4.7]{Wittenberg2015KKQp}, where the above statement was proved only \emph{away from~$p$}.\par

In the last part of this article we obtain a general transfer result for the strong $C_1^0$ property which was introduced in \cite{Wittenberg2015KKQp}. As mentioned above, the main interest of the strong $C_1^q$ property is that, in contrast with the usual $C_1^q$ property, it is amenable to transfer results along discrete henselian valuations such as \cite[Théorème~4.2]{Wittenberg2015KKQp}. However, the strong $C_1^q$ property gives interesting arithmetic information on its own. For instance, the $C_1^0$ property implies that any variety $X$ admitting a coherent sheaf $E$ such that $\chi(X, E) = 1$ —e.g. Fano varieties or rationally connected varieties— admits a zero-cycle of degree~$1$.
\begin{thmalpha}\label{thmalpha strong transfer}
    Let $(K,v)$ be a tame field with divisible value group and residue field $k$. If $k$ satisfies the strong $C_1^0$ property, then $K$ also satisfies the strong $C_1^0$ property.
\end{thmalpha}
\noindent This was previously known only for Puiseux series over a field of characteristic $0$ \cite[Proposition 7.4]{Wittenberg2015KKQp}. 
As an application of Theorem~\ref{thmalpha strong transfer}, we obtain the following corollary, answering an open question in \cite[Section~5]{Wittenberg2015KKQp}.
\begin{cor*}
    The field $\mathbf{Q}_p$ satisfies the strong $C_1^1$ property.
\end{cor*}
\noindent We also introduce a \textit{geometric} variant of the $C_1^q$ property —denoted by $C_{rc}^q$— that replaces hypersurfaces in the Fano range by smooth geometrically rationally connected varieties (see Definition~\ref{defi geometric Ciq}). This property takes inspirations from Kollar's definition of geometrically $C_1$ fields. Since homogeneous spaces are rationally connected, it also relates to the $C_{\mathrm{HS}}^q$ property concerning homogeneous spaces defined in \cite{DiegoLuco2022HomogeneousKtheory}. Regarding this property, we obtain the following result.
\begin{thmalpha}\label{thmalpha geometric properties transfer}
    Let $k$ be a field of characteristic zero. Assume that $k$ satisfies $C_{rc}^q$. Then, the Laurent series field $k\lau{t}$ satisfies $C_{rc}^{q+1}$.
\end{thmalpha}

\subsection*{Outline of proofs}
\subsubsection*{Proof sketch of Theorem \ref{thmalpha transfer intro} }
The key idea is that, in order to prove that a field $k$ is stably $C_i$, it suffices to exhibit a \emph{single} henselian valued field $(K,v)$ with residue field $k$ and divisible value group such that $K$ is $C_i$. Indeed, using the model theory of tame fields, every other such valued field has the same theory as $K$ and hence is also $C_i$. \vspace*{1ex} \\
\textbf{Part 1:} 
Suppose that $k$ is stably $C_i$. It suffices to prove the following:  
\begin{enumerate}[label=(\roman*)]
    \item if $l/k$ is algebraic, then $l$ is stably $C_i$. 
    \item If $l=k(t)^{\text{perf}}$, then $l$ is stably $C_{i+1}$. 
\end{enumerate}
Let $(K,v)$ be a tame field with divisible value group and residue field $k$, such that $k$ is $C_i$. Starting from $K$, we explain how to construct a tame field $L$ which is $C_i$ (respectively $C_{i+1}$) and has residue field $l$.  \\
(i) Consider the unramified extension $L/K$ lifting $l/k$. Since $L/K$ is algebraic and $K$ is $C_i$, the same is true for $L$. Moreover, $L$ is tame with divisible value group. \\
(ii) Start with $L_0=K(t)$, equipped with the Gauss extension $u$ of the valuation on $K$, with value group equal to the value group of $K$ and residue field $k(t)$. We can now construct an algebraic extension $L/L_0$ such that $L$ is tame and has residue field $l$. Since $L/K$ is of transcendence degree $1$, the field $L$ is $C_{i+1}$. \vspace*{1ex} \\
\textbf{Part 2:}
Using a slight generalization of a result of \cite{aug}, we prove that such a $k$ embeds elementarily into a Hahn series field $K_0=k(\!(t^{\Delta})\!)$, for a suitable ordered abelian group $\Delta$. Since $k$ is $C_i$, this implies that $K_0$ is $C_i$. If $\Delta$ is divisible, then we are done. Otherwise, we replace $K_0$ with a totally ramified algebraic extension $K$ to make the value group divisible. Since $K/K_0$ is algebraic, $K$ is $C_i$ and is also tame with divisible value group and residue field $k$.  \vspace*{1ex}
\subsubsection*{Proof sketch of Theorem \ref{thmalpha KK iterated Laurent} }
Let us discuss the main idea behind the proof of Theorem~\ref{thmalpha KK iterated Laurent}, or rather its first corollary, in the case of $k:=\mathbf{C}\lau{t_1}\lau{t_2} \lau{t_3}$.\par
The field $k$ is known to satisfy $C_3$ due to \cite[Theorem~2]{Greenberg1966RationalPointshenselian}, $C_1^2$ due to \cite[Corollaire~4.7]{Wittenberg2015KKQp} and $C_0^3$. Then the only missing property to conclude that $k$ satisfies Kato-Kuzumaki's conjecture is the $C_2^1$ property. Noting that the group $k^{\times}$ is generated by $t_1,t_2,t_3$ and a divisible subgroup, the $C_2^1$ property can easily be reduced to proving that $t_1,t_2$ and $t_3$ belong to the norm group of any hypersurface. The approach we implement in Section \ref{section applications} is to consider the fields
    \[
    k_1:= \bigcup_{n_1\geq 1} k (t_1^{1/n_1}) \qquad k_2:= \bigcup_{n_2\geq 1} k (t_2^{1/n_2}) \qquad k_3:= \bigcup_{n_1\geq 1} k (t_3^{1/n_3})
    \]
    and prove that they are $C_2$. The field $k_1$ is known to be $C_2$ because  we have $k_1= \bigcup_{n\in \mathbf{N}}\mathbf{C}\lau{t_1^{1/n}}\lau{t_2}\lau{t_3}$ and \cite[Theorem~2]{Greenberg1966RationalPointshenselian} ensures that it is a $C_2$ field. The field $k_2$ was known to be $C_2$ before our work, but it is more difficult to prove it. Indeed, the field $\mathbf{C}\lau{t_1}$ is geometrically $C_1$. Then $\bigcup_{n\in \mathbf{N}}\mathbf{C}\lau{t_1}\lau{t_2^{1/n}}$ is $C_1$ by \cite[Theorem~5.2.1]{Kartas2024geoC1}, and therefore $k_2$ is $C_2$ by \cite[Theorem~2]{Greenberg1966RationalPointshenselian}. Prior to our work, it had not been observed that $k_3$ is $C_2$, but Theorem~\ref{thmalpha transfer intro} allows us to deduce this. \par
    A similar strategy works in positive characteristic. Let $k$ be $k_0\lau{t_1}\lau{t_2}\lau{t_3}^{\mathrm{perf}}$ where $k_0$ is an algebraically closed field of positive characteristic $p$. We can construct analogous fields $k_1,k_2$ and $k_3$ such that the $p^r$-th roots of $t_1,t_2$ and $t_3$ are norms of every finite subextensions. Moreover, we can deduce that these fields are $C_2$ thanks to Theorem~\ref{thmalpha transfer intro}. It should be noted that in this case we apply Theorem~\ref{thmalpha transfer intro} to fields that are elementary equivalent to tame fields without being tame, see Example~\ref{ex t-tame not tame}.
    

\subsubsection*{Proof sketch of Theorem~\ref{thmalpha strong transfer}}
Wittenberg’s proof in the equal characteristic $0$ case relies on resolution of singularities. If such desingularization results were available in mixed and positive characteristic, the same proof would go through. Since this remains an open problem, we use tools from the model theory of valued fields as substitutes for resolution; see also the discussion in \cite[\S 1.2.1]{Kartas2024geoC1}.


Theorem~\ref{thmalpha strong transfer} can be reduced to the case where $v$ has rank one and $G_k$ is a pro-$\ell$ group for some prime $\ell$. We now sketch the proof in this situation. Let $X$ be a proper $K$-variety. We must show that the index of $X$ divides its Euler characteristic $\chi(X,\mathcal{O}_X)$. We first treat the case $\chi(X,\mathcal{O}_X)=1$. Let $\mathcal{X}$ be an arbitrary $\mathcal{O}_K$-model of $X$. Since the Euler characteristic is constant in flat families, the special fibre $\mathcal{X}_k$ also has Euler characteristic $1$. Because $k$ is strongly $C_1^0$, it follows that $\mathcal{X}_k$ admits a zero-cycle of degree $1$. But since every finite extension of $k$ has degree divisible by $\ell$, having a zero-cycle of degree $1$ is the same as having a $k$-point. 
Therefore $\Xx(k)\neq \emptyset$ for every $\Oo_K$-model. We conclude that $X(K)\neq \emptyset$ from the following result: 
\begin{fact*}[Theorem 1.2.2 \cite{Kartas2024geoC1}]
Let $(K,v)$ be a tame field with divisible value group of rank $1$ and residue field $k$. Let $X$ be a proper $K$-variety. Suppose that $\Xx(k)\neq \emptyset$, for each $\Oo_K$-model $\Xx$ of $X$. Then $X(K)\neq \emptyset$.
\end{fact*}
\noindent 
To treat the case where $\chi(X,\mathcal{O}_X)\neq 1$, we prove a slight generalization of the above result. This asserts that if there exists $d\ge 1$ such that every $\Oo_K$-model $\mathcal{X}$ has a rational point over an extension of $k$ of degree at most $d$, then $X$ has a rational point over an extension of $K$ of degree at most $d$.


\subsection*{Organization of the paper}
\begin{itemize}
    \item In Section~\ref{sec prelim and not}, we collect some preliminary results from the model theory of valued fields, the model theory of ordered abelian groups, and Milnor $\mathrm{K}$-theory.
    \item In Section~\ref{sec transfer Laurent}, we introduce the notion of a stably $C_i$ field and prove Theorem~\ref{thmalpha transfer intro} together with some basic properties of this notion.
    \item In Section~\ref{section applications}, we prove Theorem \ref{thmalpha KK iterated Laurent} and deduce Kato-Kuzumaki’s conjecture for iterated Laurent series over several arithmetically significant fields. 
    \item In Section~\ref{sec strong C1q}, we prove Theorem~\ref{thmalpha strong transfer} and deduce that $\mathbf{Q}_p$ is strongly $C_1^1$. 
    \item In Section~\ref{sec geometric variant} we also introduce the $C_{rc}^q$ property and prove Theorem~\ref{thmalpha geometric properties transfer}.
\end{itemize}

\section{Preliminaries}\label{sec prelim and not}
\subsection{Ax-Kochen/Ershov principles} \label{akesec}

Let $L$ be a language. Given $L$-structures $M$ and $N$, we use the notation $M\equiv N$ to indicate that the structures $M$ and $N$ are elementarily equivalent, the language $L$ being implicit. We write \(L_{\mathrm{rings}}=\{0,1,+,\cdot\}\) for the language of rings, \(L_{\mathrm{oag}}=\{0,+,<\}\) for the language of ordered abelian groups, and
\(L_{\mathrm{val}}=L_{\mathrm{rings}}\cup\{\mathcal O\}\) for the language of valued fields, where $\Oo$ is a unary
predicate for the valuation ring. \par 
A celebrated theorem, proved by Ax-Kochen \cite{AK12} and independently by Ershov \cite{Ershov}, states that the first-order theory of some valued fields is completely determined by its residue field and value group. Explicitly, we have the following result:
\begin{fact}[Ax-Kochen/Ershov] \label{prop Ax-Kochen/Ershov}
Let $(K,v)$ and $(K',v')$ be two henselian valued fields of equal characteristic $0$. Then,
\[(K,v)\equiv (K',v')\mbox{ in }L_{\text{val}}\iff k\equiv k'\mbox{ in }L_{\text{rings}} \mbox{ and } \Gamma\equiv \Gamma'\mbox{ in }L_{\text{oag}}\]
\end{fact}
\noindent This statement fails for general henselian valued fields of characteristic $p$, but remains true in the more restricted class of \textit{tame fields}. We refer the reader to the original article by F.-V. Kuhlmann \cite{Kuhl} for details.
\begin{defi}
Let $(K,v)$ be a henselian valued field with value group $\Gamma$ and residue field $k$. A finite valued field extension $(K',v')/(K,v)$ is said to be \textit{tame} if the following are satisfied:
\begin{enumerate}
\item  If $p=\text{char}(k)>0$, then $p\nmid [\Gamma':\Gamma]$. 
\item The residue field extension $k'/k$ is separable.
\item The extension $(K',v')/(K,v)$ is \textit{defectless}, i.e., 
$$[K':K]=[\Gamma':\Gamma]\cdot [k':k]$$
\end{enumerate}
We say that $(K,v)$ is \textit{tame} if every finite valued field extension of $(K,v)$ is tame. 
\end{defi}
\begin{rem}
\begin{itemize}
    \item[] 
    \item Every henselian field of equal characteristic $0$ is tame.
    \item Tame fields of positive characteristic are perfect. 
\end{itemize}

\end{rem}
\begin{fact} \label{akekuhl} (\cite[Theorem~1.4]{Kuhl})
Let $(K,v)$ and $(K',v')$ be two equal characteristic tame fields. Then, 
$$(K,v)\equiv (K',v')\mbox{ in }L_{\text{val}}\iff k\equiv k'\mbox{ in }L_{\text{rings}} \mbox{ and } \Gamma\equiv \Gamma'\mbox{ in }L_{\text{oag}}$$
\end{fact}
\noindent Recall that a class $\mathcal{C}$  of $L$-structures is called an \textit{elementary class} if it is the class of models of some $L$-theory $T$. The following fact is explained in \cite[Section~7]{Kuhl}. 
\begin{fact}\label{tame elemclass}
The class of tame fields is elementary in the language of valued fields. 
\end{fact}

\noindent For any cardinal $\kappa$ an $L$-structure $M$ is said to be $\kappa$-saturated (resp. $\kappa^{+}$-saturated) when for every set $\{\varphi_i(x_1,\dots, x_n)\}_{i\in I}$ of $L$-formulas in $n$-variables such that $|I| < \kappa$ (resp. $|I| \leq \kappa$) and for every $i\in I$ there exists $a_1^{(i)},\dots, a_n^{(i)} \in M$ satisfying $M \models \varphi_i(a_1^{(i)},\dots, a_n^{(i)})$, there exist $a_1,\dots, a_n \in M$ such that $M \models \varphi_i(a_1,\dots, a_n)$ for every~$i \in I$.\par 
The next result is known as the relative embedding property for tame fields:
\begin{fact}[Theorem 7.1 \cite{Kuhl}] \label{repimproved}
Let $(K_0,v_0)$ be a defectless valued field and $(K,v), (K',v')$ be two valued fields extending $(K_0,v_0)$ with $(K',v')$ tame. Let $p$ be the residue characteristic of $K_0$. Suppose the following hold: 
\begin{enumerate}
     \item $\Gamma_0$ is $p$-divisible and $k_0$ is perfect. 
     \item $\Gamma/\Gamma_0$ is torsion-free. 
     \item $(K',v')$ is $|K|^+$-saturated.
     \item There is an embedding $\rho: k\hookrightarrow k'$ over $k_0$ and an embedding $\sigma: \Gamma \hookrightarrow \Gamma'$ over $\Gamma_0$.  
\end{enumerate}
Then there exists an embedding of valued fields $(K,v)\hookrightarrow (K',v')$ inducing $\rho$ and $\sigma$.
\end{fact}
\begin{proof}
Note that \cite[Theorem 7.1]{Kuhl} also assumes that $(K,v)$ is tame but one can reduce to that case as follows. Let $(K_1,v_1)$ be an algebraic extension of $(K,v)$ such that $(K_1,v_1)$ is tame with $k_1=k^{\text{perf}}$ and $\Gamma_1=\frac{1}{p^{\infty}}\Gamma$. Since $k'$ is perfect, the embedding $\rho$ extends to an embedding $\rho_1:k_1\hookrightarrow k'$ over $k_0$. Likewise, the embedding $\sigma$ extends to an embedding $\sigma_1:\Gamma_1\to \Gamma'$ over $\Gamma_0$. Note also that $|K_1|=|K|$ since $K_1/K$ is algebraic. Now \cite[Theorem 7.1]{Kuhl} applies to give an embedding of $(K_1,v_1)\hookrightarrow (K',v')$ inducing $\rho_1$ and $\sigma_1$. Restricting to $(K,v)$, gives the desired embedding.
\end{proof}
\noindent The previous result has a direct consequence for existential theories. For a given $L$-structure $M$ with an $L$-substructure $M_0$, we denote by $\operatorname{Th}_{\exists,M_0}(M)$ the set of existential sentences (i.e. a formula only involving $\vee,\wedge$ and $\exists$) in the extended language $L(M_0)$ (i.e. $L$ with an extra constant symbol for every element of $M_0$) satisfied by $M$. Also recall that for a set of sentences $\Sigma$, we write $M \models \Sigma$ if $M$ satisfies all sentences in $\Sigma$.\par
\begin{fact} \label{existential kuhlmann}
Let $(K_0,v_0)$ be a defectless valued field and let $(K,v), (K',v')$ be two valued fields extending $(K_0,v_0)$ with $(K',v')$ tame. Assume that $\Gamma_0$ is $p$-divisible, $k_0$ is perfect and $\Gamma/\Gamma_0$ is torsion-free. Then:
$$(K',v')\models \text{Th}_{\exists,K_0} (K,v)\iff k'\models \text{Th}_{\exists,k_0} (k) \mbox{ and }\Gamma'\models \text{Th}_{\exists, \Gamma_0} (\Gamma) $$
\end{fact}
\begin{proof}
The ``only if'' direction is clear. For the converse, suppose that $k' \models \text{Th}_{\exists,k_0} (k) $ and $\Gamma' \models \text{Th}_{\exists, \Gamma_0} (\Gamma)$. 
Consider a $|K|^+$-saturated elementary extension $(K^*,v^*)$ of $(K',v')$. Since $k^*$ is $|k|^+$-saturated and $k^*\models \text{Th}_{\exists,k_0} (k)$, there is an embedding $\rho: k\hookrightarrow k^*$ over $k_0$. Likewise, there is an embedding $\sigma: \Gamma \hookrightarrow \Gamma^*$ over $\Gamma_0$. By 
Fact \ref{repimproved}, there is an embedding $\iota: (K,v)\hookrightarrow (K^*,v^*)$ over $(K_0,v_0)$ inducing $\rho$ and $\sigma$. It follows that $(K^*,v^*)\models \text{Th}_{\exists,K_0} (K,v)$. Since $(K^*,v^*)$ is an elementary extension of $(K',v')$, we conclude that $(K',v')\models \text{Th}_{\exists,K_0}(K,v)$. 
\end{proof}
\noindent We also need a ``spreading out'' lemma whose proof is similar to \cite[Lemma 5.1.3]{Kartas2024geoC1}:
\begin{lemma} \label{tame spread out}
Let $(K,v)$ be a tame field and $a_1,...,a_n\in K$. Then, there exists a tame valued subfield $(K',v')\subseteq (K,v)$ such that: 
\begin{enumerate}

\item $a_1,...,a_n\in K'$.
\item $\Gamma/\Gamma'$ is torsion-free and $k'=k$.
\item $\Gamma'$ has finite rank. 
\end{enumerate}
\end{lemma}
\begin{proof}
Let $F$ be the prime subfield of $K$ (either $\mathbf{Q}$ or $\mathbf{F}_p$) and consider $K_0=F(a_1,...,a_n)$ equipped with the restriction $v_0$ of $v$. By Abhyankar's inequality \cite[Proposition~2, Appendix~2]{ZariskiSamuel1976CommutativeAlgebra}, the rank of $\Gamma_0:=v_0(K_0^{\times})$ is at most $\operatorname{trdeg}(K_0/F)+1$. Denote by $k_0$ the residue field of $(K_0,v_0)$.
Choose a transcendence basis $\tau=\{t_i:i\in I\}$ for $k/k_0$. Let $T_i$ be a lift of $t_i$ in $K$ and $\mathrm{T}=\{T_i:i\in I\}$. We now let $K_1=K_0(\mathrm{T})$ in case $(K,v)$ is of mixed characteristic, and $K_1=K_0(\mathrm{T} \cup \{t\})$ in case $(K,v)$ is of equal characteristic, where $t$ is any element in the maximal ideal of $K$. In each case, we endow $K_1$ with the restriction $v_1$ of $v$.
By \cite[Lemma~2.2]{Kuhl}, we get that $(K_1,v_1)$ is a valued field with residue field $k_0(\tau)$ and the same value group as $K_0$. 
Finally, let $K'\subseteq K$ be the relative algebraic closure of $K_1$ in $K$. We endow $K'$ with the restriction $v'$ of $v$. By \cite[Lemma~3.7]{Kuhl}, the field $(K',v')$ satisfies (i) and (ii). We also have that $\Gamma':=v'(K'^{\times})$ has the same rank as $\Gamma_1:=v_1(K_1^{\times})$ since $K'/K_1$ is algebraic.    
\end{proof}

\subsection{Totally ordered abelian groups and Hahn series}
\noindent Let $\Gamma$ and $\Delta$ be ordered abelian groups. We endow their direct sum $\Delta \oplus \Gamma $ with the lexicographic order defined by 
\[
(\delta,\gamma) <_{\Delta \oplus \Gamma} (\delta',\gamma')
\quad\Longleftrightarrow\quad
\big(\delta <_\Delta \delta'\big)
\ \text{ or }\
\big(\delta = \delta' \text{ and } \gamma <_\Gamma \gamma'\big)
\]
Recall that the symbol $\preceq$ is used to denote ``elementary substructure''.
\begin{proposition}[Theorem 4.3 \cite{aug}] 
For any ordered abelian group $\Gamma$, there is a non-trivial ordered abelian group $\Delta$ such that $\Gamma \preceq \Delta \oplus \Gamma$. 
\end{proposition}
\noindent This result is relevant for us because combined  with Ax-Kochen/Ershov's theorem, it proves that $t$-henselian fields (Definition~\ref{def thenselian ttame} below) of characteristic $0$ are an elementary substructure of a suitable Hahn series field over $k$. We recall the definition of Hahn series.
\begin{defi}
For an ordered abelian group $\Gamma$ and a field $k$, the Hahn series over $k$ with exponents in $\Gamma$ is the set
\[
k(\!(t^\Gamma)\!) \ :=\ \Bigl\{\, \sum_{\gamma\in S} a_\gamma t^\gamma \ \Bigm|\ 
S\subseteq \Gamma\ \text{is well-ordered and } a_\gamma\in k \backslash{\{0\}} \Bigr\}
\]
With termwise addition and Cauchy product, this is a field equipped with a henselian valuation
$v\!\left(\sum_{\gamma \in S} a_\gamma t^\gamma\right)=\min S$, having value group $\Gamma$ and residue field $k$. 
\end{defi}
\begin{proposition}[Proposition 5.1 \cite{aug}] \label{aug}
Let $k$ be a $t$-henselian field of characteristic $0$. Then, $k\preceq k(\!(t^{\Delta})\!)$ for some non-trivial ordered abelian group $\Delta$.
\end{proposition}
\noindent This result extends to positive characteristic provided that the field is $t$-tame.
\begin{proposition} \label{augimproved}
Let $k$ be a $t$-tame field. Then $k\preceq k(\!(t^{\Delta})\!)$ for some non-trivial ordered abelian group $\Delta$. Moreover, for any such $\Delta$, the field $(k(\!(t^{\Delta})\!),v_t)$ is tame.
\end{proposition} 
\begin{proof}
One proceeds as in the proof of \cite[Proposition 5.1]{aug}, replacing every occurrence of Ax-Kochen/Ershov's theorem with Kuhlmann's theorem in the positive characteristic case. For the second part of the statement, recall that tame fields are perfect and therefore  $k$ and $k(\!(t^{\Delta})\!)$ must both be perfect. Since $k(\!(t^{\Delta})\!)$ is perfect, this forces $\Delta$ to be $p$-divisible. Since Hahn series are henselian and defectless, we conclude that $(k(\!(t^{\Delta})\!),v_t)$ is a tame field.
\end{proof}

\subsection{Topologically henselian fields}
We recall the notion of a $t$-henselian field as defined in \cite[\S~7]{zieglerprestel} as well as a tame variant which is discussed in \cite[\S~4]{Ans}.
\begin{defi}\label{def thenselian ttame}
A field $k$ is called $t$-henselian if it is elementarily equivalent to a field which admits a non-trivial henselian valuation. If in addition the valuation is tame, we say that $k$ is $t$-tame. 
\end{defi}
\begin{example}
The field $\mathbf{R}$ does not admit any non-trivial henselian valuation but is $t$-henselian since it is elementarily equivalent to $\bigcup_{n\in \mathbf{N}} \mathbf{R}(\!(t^{1/n})\!)$.
\end{example}

\noindent For all applications in this paper valuations are of finite rank. However, the results in this section hold for a wider family of valuations:
\begin{defi}
Let $K$ be a field and $v$ a valuation on $K$. The valuation $v$ is said to be \textit{microbial} if it admits a rank $1$ coarsening. 
\end{defi}
\noindent The next result shows that $t$-henselian fields are not very far from being $t$-tame. This is a special case of \cite[Proposition 4.2.2]{JK}, but we explain the argument for the convenience of the reader. 
\begin{proposition}\label{prop perfect microbial is t-tame}
Let $(K,v)$ be a perfect henselian valued field. In case of positive characteristic, assume also that $v$ is microbial. Then $K$ is $t$-tame.
\end{proposition} 
\begin{proof}
We deal first with the case \(\mathrm{char}(K)=0\). We may assume without loss of generality that $K$ admits a henselian valuation $v$.  Note that any non-principal ultrapower $K^{\mathcal{U}}$ of \((K,v)\) admits a henselian valuation of  residue characteristic \(0\). This is clear if $v$ has residue characteristic $0$. If $v$ has residue characteristic $p$, then $\Oo_{K^{\mathcal{U}}}[1/p]$ is a non-trivial coarsening of $\Oo_{K^{\mathcal{U}}}$ of residue characteristic $0$ and also henselian. In particular, the ultrapower $K^{\mathcal{U}}$ is tame. Since $K$ is elementarily equivalent to its ultrapower, we deduce that \(K\) is \(t\)-tame. \par

Now assume that $\text{char}(K)=p$ is positive and $v$ is microbial. Since $v$ is microbial, there is \(t\in\Oo_K\) such that $\Z \cdot v(t)$ is cofinal in the value group. 

\medskip
\noindent\emph{Claim.} Every finite separable extension of \(K\) is generated by a monic \(f\in\Oo_K[X]\) with
\(0\le v(\mathrm{disc}(f))\le v(t)\) where $\mathrm{disc}(f)$ denotes the discriminant of $f$.

\smallskip
\noindent Start with any monic separable \(g\in\Oo_K[X]\) generating the extension. Since \(\mathbf Z\cdot v(t)\) is cofinal in the value group, there is \(N\in \mathbf{N}\) such that 
$$ 0\leq v(\mathrm{disc}(g))\le N\,v(t)$$
Since \(K\) is perfect, taking \(p\)-th roots of the coefficients of $g$ gives a monic polynomial $g'$ that generates the same extension as $g$ and we have $\mathrm{dics}(g') = \frac{1}{p}\mathrm{disc}(g)$. After repeating this process finitely many times, we obtain \(f\) with \(0\leq v(\mathrm{disc}(f))\le v(t)\). This is enough to prove the claim.

The property "every finite extension 
is generated by a monic polynomial $f$ such that $0 \leq v(\mathrm{disc}(f)) \leq v(t)$" 
is elementary in the language of valued fields with a parameter $t$ (one sentence per degree). Then, it holds for any non-principal ultrapower \((K^{\mathcal{U}},v^{\mathcal{U}})\). Let $\Oo_w=\Oo_{K^{\mathcal{U}}}[t^{-1}]$. From the above discussion, we see that every finite extension of $(K^{\mathcal{U}},w)$ generated by a monic polynomial $f\in \Oo_w[X]$ such that $ w(\mathrm{disc}(f))=0$ and therefore comes from a separable residue field extension of the same degree. In particular, we get that   $(K^{\mathcal{U}},w)$ is tame and hence $K$ is $t$-tame.
\end{proof}

\begin{example}\label{ex t-tame not tame}
There are perfect henselian fields which are not tame, but only $t$-tame. For example, the field $K:=\mathbf{F}_p(\!(t)\!)^{\text{perf}}$ is $t$-tame by the above result but is not tame since the Artin-Schreier equation
$$X^p-X-1/t=0 $$
defines a finite immediate extension —i.e. a non-trivial extension with trivial residual extension and the same value group as $K$. 
\end{example}

\subsection{Milnor \texorpdfstring{$\mathrm{K}$}{K}-theory} 
Let $k$ be a field and $q$ a non-negative integer. We define the $q$-th Milnor $\mathrm{K}$-group of $k$ as follows: for $q=0$ we set $\mathrm{K}_0(k) = \mathbf{Z}$ and for $q \geq 1$ we set 
    \[\mathrm{K}_q(k) := \left(k^{\times} \right)^{\otimes q} \; \big/ \; \left\langle a_1 \otimes \cdots \otimes a_q \, | \, \exists \; i \neq j,\; a_i + a_j = 1 \right\rangle\]
    where the tensor product is taken over $\mathbf{Z}$. For $a_1, \cdots, a_q \in k^{\times}$ we denote by $\{a_1, \cdots, a_q\}$ the image of $a_1 \otimes \cdots \otimes a_q$ in $\mathrm{K}_q(k)$. Elements of this form are called \textit{symbols} and they manifestly generated $\mathrm{K}_q(k)$.\par
    For any pair of non-negative integers $p,q$ there is a natural pairing
    \[\{\cdot, \cdot\}: \mathrm{K}_p(k) \times \mathrm{K}_q(k) \to \mathrm{K}_{p+q}(k)\]
    induced by the tensor pairing $\left(k^{\times} \right)^{\otimes p} \times \left(k^{\times} \right)^{\otimes q} \to \left(k^{\times} \right)^{\otimes (p+q)}$. \par
    For every field extension $K/k$ there exists a natural \textit{restriction} map $\iota_{K/k}:\mathrm{K}_q(k) \to \mathrm{K}_q(K)$. We usually denote $\iota_{K/k}(\alpha)$ by $\alpha|_K$ for $\alpha\in \mathrm{K}_q(k)$. \par
    Let $l/k$ be a finite extension. One can construct a norm homomorphism
    \[\mathrm{N}_{l/k}: \mathrm{K}_q(l) \to \mathrm{K}_q(k)\]
    satisfying the following properties
    \begin{enumerate}
        \item For $q = 0$, the map $\mathrm{N}_{l/k}: \mathbf{Z}\to \mathbf{Z}$ is given by multiplication by $[l:k]$,
        \item For $q = 1$ the map $\mathrm{N}_{l/k}: \mathrm{K}_1(l) \to \mathrm{K}_1(k)$ coincides with the usual norm $l^{\times} \to k^{\times}$,
        \item For any pair of non-negative integers $p,q$, we have $\mathrm{N}_{l/k} (\{\alpha|_l, \beta\}) = \{\alpha, \mathrm{N}_{l/k}(\beta) \}$ for $\alpha \in \mathrm{K}_p(k)$  and $\beta \in \mathrm{K}_q(l)$, and
        \item If $m$ is a finite extension of $l$ we have $\mathrm{N}_{m/k} = \mathrm{N}_{l/k} \circ \mathrm{N}_{m/l}$.
    \end{enumerate}
    The construction can be found in \cite[Section 1.7]{Kato1980GeneralizationClassFieldK2} and \cite[Section 7.3]{GilleSzamuely2017CentralGalois}. \par
    Let $K$ be a henselian discrete valued field with ring of integers $R$, maximal ideal $\mathfrak{m}$ and residue field $k$. Then for every strictly positive integer $q$ there exists a unique~morphism
    \[ \partial : \mathrm{K}_q(K) \to \mathrm{K}_{q-1}(k)\]
    such that for every uniformiser $\pi$ and units $u_2,\hdots u_q \in R^{\times}$ we have
    \[\partial(\{\pi, u_2, \hdots , u_q \}) = \{\overline{u_2}, \hdots , \overline{u_q}\}\]
    where $\overline{u_2}, \hdots , \overline{u_q}$ denotes the images of $u_2,\hdots u_q$ in the residue field. We call this map the \textit{residue map}. \par
    We denote the kernel of $\partial: \mathrm{K}_q(K) \to \mathrm{K}_{q-1}(k)$ by $\mathrm{U}_q(K)$. It is generated by symbols of the form $\{u_1 , \hdots ,u_q \}$ where $u_1, \hdots ,u_q$ are units in $R$. One can define a specialisation map $s: \mathrm{U}_q(K) \to \mathrm{K}_{q}(k)$ characterised by $s(\{u_1 , \hdots ,u_q \}) = \{\overline{u_1}, \hdots , \overline{u_q}\}$, see \cite[Proposition 7.1.4]{GilleSzamuely2017CentralGalois}. Denote by $\mathrm{U}_q^1(K)$ the kernel of $s$. It is generated by symbols of the form $\{u, a_2, \hdots, a_q\}$ where $u$ belongs to $1 + \mathfrak{m}$ and $a_2,\hdots, a_q\in K^{\times}$, see \cite[Proposition~7.1.7]{GilleSzamuely2017CentralGalois}. \par
    Moreover, the residue map and specialisation map are compatible with the norm map. Indeed, if $L/K$ is a finite extension with ramification index $e$ and the residue field of $L$ is $l$, we have the following commutative diagrams whose horizontal arrows are isomorphisms
    \begin{eqnarray*}
        \begin{tikzcd}
            \mathrm{K}_q(L)/ \mathrm{U}_q(L) \ar[r, "\sim" ', "\partial_L"]  \ar[d,"\mathrm{N}_{L/K}"]& \mathrm{K}_{q-1}(l) \ar[d,"\mathrm{N}_{l/k}"] & \mathrm{U}_q(L)/ \mathrm{U}_q^1(L) \ar[r, "\sim" ', "s_L"]  \ar[d,"\mathrm{N}_{L/K}"]& \mathrm{K}_{q}(l) \ar[d,"e \cdot \mathrm{N}_{l/k}"] \\
            \mathrm{K}_q(K)/ \mathrm{U}_q(K) \ar[r, "\sim" ', "\partial_K"] & \mathrm{K}_{q-1}(k)  & \mathrm{U}_q(K)/ \mathrm{U}_q^1(K) \ar[r, "\sim" ', "s_K"]  & \mathrm{K}_{q}(k).
        \end{tikzcd}
    \end{eqnarray*}
\subsection*{The \texorpdfstring{$C_i^q$}{Kato-Kuzumaki} properties}
    Let $k$ be a field, $i,q$ non-negative integers, and $X$ a $k$-scheme of finite type. We define the $q$\textit{-th norm group} of $X$ as 
    \[\mathrm{N}_q(X/k) := \left\langle \mathrm{N}_{k(x)/k}\left(\mathrm{K}_q(k(x))\right) \; | \; x \in X_{(0)}\right\rangle \subseteq \mathrm{K}_q(k)\]
    where $X_{(0)}$ denotes the set of closed points of $X$. When $l/k$ is a finite extension we write $\mathrm{N}_q(l/k)$ for $\mathrm{N}_{q}(\Spec (l) /k)$. \par
    The field $k$ is said to have the $C_i^q$ property if for every non-negative integers $n$ and $d$, finite extension $l/k$ and hypersurface $Z$ of $\mathbf{P}^n_l$ of degree $d$ with $d^i \leq n$ we have $\mathrm{N}_q(Z/l) = \mathrm{K}_q(l)$.

\subsection*{Kato-Kuzumaki's conjecture away from \texorpdfstring{$p$}{p}}
\noindent We mimic the definition of the strong $C_1^q$ away from $p$ given in \cite[Définition~4.1]{Wittenberg2015KKQp}.
\begin{defi}
    Let $k$ be a field of positive characteristic $p$. The field $k$ is said to satisfy the $C_i^q$ property away from the characteristic if for every finite extension $l/k$ and hypersurface $Z \subseteq \mathbf{P}^n_l$ of degree $d$ such that $d^i \leq n$ the quotient
    \[ \mathrm{K}_q(l) / \mathrm{N}_q(Z/l)\]
    is a $p$-primary torsion group.
\end{defi}
\noindent In the application of this paper we are able to prove a stronger property.
\begin{defi}
    We say that $k$ satisfies $C_i^q$ away from the imperfection if its perfection $k^{\mathrm{perf}}$ satisfies $C_i^q$.
\end{defi}
\noindent A standard restriction-correstriction argument proves that if a field satisfies $C_i^q$ away from the imperfection, it also satisfies $C_i^q$ away from the characteristic. Moreover, we have the following lemma.
\begin{lemma}\label{lemma away from imperf implies property}
    Let $k$ be a field of positive characteristic $p$ and $[k:k^p] \leq p^{\delta}$. If $k$ satisfies $C_i^q$ away from the imperfection and $q \geq \delta$, then $k$ satisfies $C_i^q$.
\end{lemma}
\noindent This follows directly from \cite[Proposition A.4]{HarryFelipe2025Stability}. We do not know whether the previous lemma holds replacing ``imperfection'' by ``characteristic''.


\section{Stably \texorpdfstring{$C_i$}{Ci} fields}\label{sec transfer Laurent}

\subsection{Basic properties}
\begin{defi}
    Let $k$ be a perfect field and $i\in \mathbf{N}$. We say that $k$ is stably $C_i$ if every equicharacteristic tame field $(K,v)$ with divisible value group and residue field $k$ is $C_i$.
\end{defi}
\noindent 
Note that a stably $C_i$ field is $C_i$, since the residue field of a $C_i$ valued field is also $C_i$. 
We can also replace ``every'' by ``some'' in the above definition:
\begin{lemma} \label{every = some}
A field $k$ is stably $C_i$ if and only if there exists a non-trivially valued tame field $(K,v)$ with divisible value group and residue field $k$ such that $K$ is $C_i$.
\end{lemma}
\begin{proof}
The ``only if'' is clear. For the converse, using Fact \ref{akekuhl}, any two non-trivially valued tame fields with divisible value group and residue field $k$ have the same theory. This is enough to conclude because the $C_i$ property can be expressed in the language of rings, using one sentence for each degree $d$ and $n$ such that $d^i \leq n$.
\end{proof}

\begin{lemma} \label{colimits}
The class of stably $C_i$ fields is closed under colimits.    
\end{lemma}
\begin{proof}
Let $(k_j)_{j\in J}$ be a directed family of stably $C_i$ fields and $k:=\underrightarrow{\operatorname{colim}}\;  k_j$. For each $j\in J$, we set $K_j$ to be the Hahn series field $K_j=k_j(\!(t^{\mathbf{Q}})\!)$. Note that $(K_j)_{j\in J}$ naturally forms a directed system. Then, the colimit $K=\underrightarrow{\operatorname{colim}}\;  K_i$ endowed with the $t$-adic valuation is a tame field with divisible value group and residue field $k$ and it is $C_i$ because the class of $C_i$ fields is closed under colimits. We conclude that $k$ is stably $C_i$ from Lemma~\ref{every = some}.  
\end{proof}

\begin{lemma}
 The class of stably $C_i$ fields is an elementary class in the language of rings.  
 Moreover, it is $\forall \exists$-axiomatizable.  
\end{lemma}
\begin{proof}
 By \cite[Theorem 4.1.12]{changkeisler}, it suffices to prove that stably $C_i$ fields are closed under ultraproducts and elementary equivalence. Let $k=\prod_{j\in J} k_j/U$ be an ultraproduct of stably $C_i$ fields. For each $j\in J$, let $(K_j,v_j)$ be a non-trivial valued tame field with divisible value group and residue field $k_j$. Note that $K_j$ is $C_i$ for every $j\in J$ because $k_j$ is stably $C_i$. The class of $C_i$ fields is an elementary class, and the same is true for the class of tame fields by Fact \ref{tame elemclass}. Therefore, the ultraproduct $(K,v)=\prod_{j\in J} (K_j,v_j)/U$ is a tame field with divisible value group and residue field $k$ and it is $C_i$. By Lemma~\ref{every = some}, we deduce that $k$ is stably $C_i$. \par
 Let $k$ be a stably $C_i$ field an $k'$ a field that is elementary equivalent to $k$. Let $(K,v)$ and $(K',v')$ be a tame equicharacteristic fields with divisible value group, and $k$ and $k'$ as residue field respectively. By definition of stably $C_i$, the field $K$ is $C_i$. On the other hand, $(K,v)$ and $(K,v')$ are elementary equivalent thanks to Fact~\ref{akekuhl}. We conclude that $K'$ is $C_i$ because it is an elementary property.\par
 For the moreover part, we have that stably $C_i$ fields are closed under directed unions by Lemma \ref{colimits}, and \cite[Corollary 3.1.9]{tent} says that every elementary class closed under directed unions is $\forall \exists$-axiomatizable.
\end{proof}


\subsection{Transition theorems}
We prove stable analogues of the Lang-Nagata transition theorems for $C_i$ fields.
\begin{lemma} \label{algebraic}
Every algebraic extension of a stably $C_i$ field is stably $C_i$.   
\end{lemma}
\begin{proof}
Let $k$ be a stably $C_i$ field and $l/k$ be an algebraic extension. Let $(K,v)$ be non-trivially valued tame field with divisible value group and residue field $k$ which is $C_i$. Let $L/K$ be the unramified extension lifting $l/k$. The field $L$ is $C_i$ due to \cite[Theorem~2a]{Nagata1957NotesPaperLang}, it has divisible value group, and is also tame by \cite[Lemma~2.17(b)]{Kuhl}. We conclude that $l$ is stably $C_i$ from Lemma~\ref{every = some}.
\end{proof}
The next result follows the same strategy as in \cite[Proposition 4.3]{KartasPerfectoid}:
\begin{proposition} \label{stablelangnagata}
Let $k$ be a stably $C_i$ field. Then the perfection of $k(t)$ is stably $C_{i+1}$.    
\end{proposition}
\begin{proof}
Let $(K,v)$ be a tame field with divisible value group and residue field $k$, such that $K$  is $C_i$. We equip the rational function field $K(t)$  with the Gauss valuation extending the valuation on $K$. By \cite[Lemma 2.6]{KartasPerfectoid}, there is an algebraic extension $K'$ of $K(t)$ such that $(K',v')$ is a tame with residue field the perfection of $k(t)$. By Lemma~\ref{every = some}, we conclude that the perfection of $k(t)$ is stably $C_i$.
\end{proof}
\noindent Lemma~\ref{algebraic} and Proposition~\ref{stablelangnagata} yield the following corollary.
\begin{cor}\label{cor stable Lang-Nagata}
    Let $k$ be a stably $C_i$ field and $l/k$ be a field extension such that $j:=\operatorname{trdeg}(l/k)$ is finite and $l$ is perfect. Then, the field $l$ is stably $C_{i+ j}$.
\end{cor}
\begin{proposition} \label{proposition transfer function field over finite residue}
Let $k$ be a perfect field of positive characteristic $p$ of absolute transcendence degree $i$, that is $\operatorname{trdeg}(k/\mathbf{F}_p) = i$. Then $k$ is stably $C_{i+1}$.
\end{proposition} 
\begin{proof}
By \cite[Theorem 1.1.2]{Kartas2024geoC1}, finite fields are stably $C_1$. We conclude directly from Corollary~\ref{cor stable Lang-Nagata} and Lemma~\ref{algebraic}. 
\end{proof}


\subsection{The \texorpdfstring{$C_i$}{Ci} property for \texorpdfstring{$t$}{t}-henselian fields}

\begin{proposition} \label{proposition transfer Puiseux}
Every $t$-tame field $k$ which is $C_i$ is also stably $C_i$.
\end{proposition} 
\begin{proof}
By Proposition \ref{augimproved}, there is a non-trivial ordered abelian group $\Delta$ such that $k\preceq k(\!(t^{\Delta})\!)$ and $(k(\!(t^{\Delta})\!),v_t)$ is a tame field. In particular,  we get that $ k(\!(t^{\Delta})\!)$ is also $C_i$. 
Let $K$ be a maximal totally ramified extension of $ k(\!(t^{\Delta})\!)$ with respect to the $t$-adic valuation. We see that $K$ is $C_i$ thanks to \cite[Theorem~2a]{Nagata1957NotesPaperLang}. Moreover, the valued field $(K,v_t)$ is tame due to \cite[Lemma 2.17(b)]{Kuhl}. The value group of $K$ is $\Delta_{\Q}=\Delta \otimes_{\Z} \Q$, namely the divisible hull of $\Delta$. We conclude that $k$ is stably $C_i$ by Lemma \ref{every = some}.
\end{proof}  
\noindent Using the above result, we obtain a stable analogue of \cite[Theorem~2]{Greenberg1966RationalPointshenselian}.

\begin{proposition}\label{prop stable greenberg general}
    Let $k$ be a $C_i$ field and $K$ be a henselian discrete valued field with residue field $k$. Then, the perfection of $K$ is stably $C_{i+1}$.
\end{proposition}
\begin{proof}
One can prove that $K^{\mathrm{perf}}$ is $C_{i+1}$ via a similar argument to the one of \cite[Theorem~2]{Greenberg1966RationalPointshenselian}; we repeat the argument for the sake of the reader. Let $Z \subseteq \mathbf{P}_{K^{\mathrm{perf}}}^n$ a hypersurface of degree $d$ such that $d^{i+1} \leq n$ and $t \in K$ be a uniformiser. Denote by $L_0$ the henselianisation of $k(t)^{\mathrm{perf}}$ with respect to the $t$-adic valuation. For every $\nu \in \frac{1}{p^{\infty}}\mathbf{N}$ we can find a hypersurface $Z^{(\nu)} \subseteq \mathbf{P}^n_{L_0}$ with the same reduction modulo $t^{\nu}$ as $Z$. The field $L_0$ is $C_{i+1}$ because it is an algebraic extension of a $C_{i+1}$ field $k(t)$. We deduce that $Z$ admits a rational point thanks to \cite[Corollary~1.2.2]{Moret_Bailly2021valuativeGreenberg}. It follows that $K^{\mathrm{perf}}$ is $C_{i+1}$. Finally, $K^{\mathrm{perf}}$ is $t$-tame thanks to Proposition~\ref{prop perfect microbial is t-tame}, and we conclude that it is stably $C_{i+1}$ from Proposition~\ref{proposition transfer Puiseux}.
\end{proof}
\begin{rem}
We do not know an example of a $C_i$ field which is not stably $C_i$.
A folklore conjecture says that for a $C_i$ field $k$ of characteristic $0$, the Puiseux series  $\bigcup_{n=1}^{\infty} k(\!(t^{1/n})\!)$ is also $C_i$. This implies that $k$ is stably $C_i$,  since any non-trivially valued henselian field with divisible value group and residue field $k$ is elementarily equivalent to the Puiseux series over $k$ by Ax-Kochen/Ershov. Some evidence for $i=1$ is given in \cite{Kartas2024geoC1}.
\end{rem}
\section{Applications to Kato-Kuzumaki's conjecture}\label{section applications}

In this section we study the applications of the transfer result in Section \ref{sec transfer Laurent} to Kato-Kuzumaki's conjecture for iterated Laurent series. First, we introduce some notation.
\begin{defi}\label{def generating mod div and universal norms}
    Let $k$ be a field and $q$ a non-negative integer. Let us fix the following notation:
    \begin{itemize}
        \item A subset $\Delta \subseteq \operatorname{K}_q(k)$ is said to be a \textit{generating set modulo d} if its image in $\operatorname{K}_q(k)/d$ generates $\operatorname{K}_q(k)/d$ for every non-zero integer $d$.
        \item Let $K/k$ be a field extension. A class $\alpha \in \operatorname{K}_q(k)$ is said to be a \textit{universal norm} of $K/k$ if for every finite subextension $k'/k$ of $K/k$ we have $\alpha \in \operatorname{N}_q(k'/k)$.
    \end{itemize}
\end{defi}
\noindent 
We now introduce a strengthening of the $C_i^q$ property.
\begin{defi}\label{def stably uniform Ciq property}
    Let $k$ be a field and $i,q\in \mathbf{N}$. We say that $k$ satisfies the \textit{stably uniform} $C_i^q$ \textit{property} if for every finite extension $k'/k$ there exists $\Delta \subseteq \operatorname{K}_q(k')$ which is a generating set modulo $d$ for every $d\in \mathbf{N}$ and such that for every $\alpha \in \Delta$ there exists an algebraic extension $K_{\alpha}/k'$ where the field $K_{\alpha}$ is stably $C_i$ and $\alpha$ is a universal norm of $K_{\alpha}/k'$.
\end{defi}
\noindent Note that the (stably) uniform $C_i^0$ property is implied by the (stable) $C_i$ property and the (stably) uniform $C_0^q$ property is equivalent to $C_0^q$. The first non-trivial example are function fields over an algebraically closed field, this is proven in \cite[Theorem~2.2]{Diego2018KK}. We recall the argument below, see Proposition~\ref{prop KK function fields}.

\subsection{General transfer result}

\begin{thm}\label{thm general transfer}
    Let $k$ be a perfect field and $i,q$ a pair of positive integers. Assume that $k$ satisfies the stably uniform $C_{i-1}^q$ and $C_i^{q-1}$ property. Let $R_0$ be an equicharacteristic henselian discrete valued field with residue field $k$ and $K$ the fraction field of $R_0$. Then, the field $K$ satisfies the stably uniform $C_i^q$ property away from the~imperfection.
\end{thm}

\begin{rmk}
    A similar transfer principle is expected for the usual Kato-Kuzumaki's conjectures. However, we cannot just assume that $k$ satisfies $C_{i-1}^q$. Indeed, assume that the $C_{0}^1$ property for $k$ implies the $C_1^1$ property for $k\lau{t}$. If $k_0$ is the example of a field with cohomological dimension $1$ that is not $C_1^0$ given in \cite{CTMadore2004NotC10}, then such a transfer result together with \cite[Lemma~5.8]{HarryFelipe2025Stability} implies that $k_0$ satisfies $C_1^0$ which is a~contradiction. 
\end{rmk}

\begin{cor}
    Let $k$ be a perfect field satisfying the stably uniform $C_i^q$ property for every pair of non-negative integers $i,q$ such that $i+q \geq n$. Then $k\lau{t_1}\dots\lau{t_m}$ satisfies the $C_i^q$ property away from the imperfection for every pair of non-negative integers $i,q$ such that $i+q \geq n+ m$
\end{cor}
\noindent Note that the cases $i=0$ and $q=0$ in this corollary follow from \cite[Theorem~2]{Greenberg1966RationalPointshenselian}, \cite[Theorem~2a]{Nagata1957NotesPaperLang}, and \cite[Theorem~1~(1)]{KK1986DimensionOfFields}. \par
\noindent Let us replace $K$ by its perfection for the rest of the section . Before proving Theorem~\ref{thm general transfer}, we need the following lemmas.
\begin{lemma}\label{lemma generators Kthry}
        Let $t$ be a uniformiser of $R_0$ and $\mathcal{O}_K$ the ring of integers of $K$. For every non-negative integer $q$ the group $\mathrm{K}_q(K)$ is generated by symbols of the form $\{t^{1/p^{r}},u_{2},\dots , u_{q}\}$ and $\{u_1,\dots, u_q\}$ where $r$ is a positive integer and $u_{1},\dots,u_q$ are units in $K$, i.e. belong to $R^{\times}$. Moreover, for any positive integer $d$ the quotient $\mathrm{K}_q(K)/d$ is generated by symbols of the form $\{t^{1/p^r},u_{2}^{(0)},\dots , u_{q}^{(0)}\}$ and $\{u_{1}^{(0)},\dots , u_{q}^{(0)}\}$ where $u_{1}^{(0)}, \dots, u_q^{(0)}$ are elements of $k$.
    \end{lemma}
    \begin{proof}
        The first statement is a direct consequence of the fact that the multiplicative group $K^{\times}$ is generated by $\mathcal{O}_K^{\times}$ and $t^{1/p^{r}}$. \par
        For the second statement, consider a symbol $\{t^{1/p^{r}},u_{2},\dots , u_{q}\}$ or $\{u_1,\dots, u_q\}$ as in the first statement and denote by $u_1^{(0)},\dots, u_q^{(0)} \in k$ the reduction of $u_1,\dots, u_q$ modulo $\mathfrak{m}_K$. Note that for every $j\in \{1,\dots , q\}$ the product $u_j^{-1}u_j^{(0)}$ belongs to $1 + \mathfrak{m}_K$. Let $s \in \mathbf{N}$ be such that $d =p^s d'$ with $p$ and $d'$ coprime. We can apply Hensel's lemma to find $y \in \mathcal{O}_K$ such that  $u_r^{-1}u_r^{(0)} = y^{d'}$ because $\mathcal{O}_K$ is a henselian local ring \cite[Proposition~3.5~(i)]{Morrow2012Higherlocal}. Since $K$ is perfect, there exists an element $z \in \mathcal{O}_K$ such that $z^{p^s} = y$. Then, we conclude because 
        \begin{align*}
            \{t^{1/p^s},u_{2},\cdots , u_{q}\} &= \{t^{1/p^s},u_{2}^{(0)},\cdots , u_{q}^{(0)}\} & \mathrm{mod} \; d \cdot \mathrm{K}_q(K) \\
            \{u_{1},\cdots , u_{q}\} &= \{u_{1}^{(0)},\cdots , u_{q}^{(0)}\} & \mathrm{mod} \; d \cdot \mathrm{K}_q(K) 
        \end{align*}
    \end{proof}
    \begin{lemma}\label{lemma algebraically maximal extension t norm}
    Let $F$ be the perfection of a henselian discrete valued field of positive characteristic $p$. Then, for every uniformiser $t \in F$ there exists a maximal totally ramified extension $E/F$ such that for every $r \in \mathbf{N}$ and finite subextension $L/F$ of $E/F$, we have $t^{1/p^r}\in \mathrm{N}_{L/F}(L^{\times})$. Moreover, for any such extension the field $E$ is a tame field.
    \end{lemma}
    \begin{proof}
    Let $t\in F$ be a uniformiser. Consider $F'=\bigcup_{p\nmid n}F((-t)^{1/n})$ and $E$ be a maximal immediate algebraic extension of $F'$. The valued field $(E,v_t)$ is tame due to \cite[Corollary~3.12~b)]{Kuhl}. Thanks to \cite[Corollry~3.1 and Corollary~3.9]{Kuhl}, every finite subextension $L/F$ of $E/F$ is a $p^m$-extension of $F((-t)^{1/n})$ for some $m,n \in \N$ with $p\nmid n$. Since $F$ is perfect, the same is true for $L$ and therefore $(-t)^{1/np^m}\in L$. We then have that 
     \begin{align*}
              \mathrm{N}_{L/F}((-t)^{1/np^{m+r} })^{p^r} &=\mathrm{N}_{F((-t)^{1/n})/F}\circ \mathrm{N}_{L/F((-t)^{1/n})}(t^{1/np^{m} }) \\
              & = \mathrm{N}_{F((-t)^{1/n})/F}((-t)^{1/n})\\
              &=(-1)^{n}(t^{1/p^r})^{p^r}
     \end{align*}
    Since the Frobenius is injective, we conclude that $ \mathrm{N}_{L/F}((-t)^{1/np^{m+r} }) = (-1)^{n}t^{1/p^r}$. This is enough when $n$ is even or $p=2$. When $n$ is odd and $p\neq 2$, we conclude because
    \[
    \mathrm{N}_{L/F}(-(-t)^{1/np^{m+r} })= (-1)^{np^m + n}t^{1/p^r} = t^{1/p^r}.
    \]
\end{proof}
\begin{proof}[Proof of Theorem~\ref{thm general transfer}]
    Let $L/K$ be a finite extension. Note that $L$ is also the perfection of the fraction field of a henselian discrete valued field and the residue field $l$ of $L$ also satisfies the stably uniform $C_{i-1}^q$ and stably uniform $C_i^{q-1}$ property because $l/k$ is a finite extension. We may and do assume without loss of generality that $L=K$. \par
    
    Let $\Delta_{q-1} \subseteq \mathrm{K}_{q-1}(k)$ and $\Delta_{q} \subseteq \mathrm{K}_q(k)$ be the set given by the definition of the stably uniform $C_i^{q-1}$ and stably uniform $C_{i-1}^q$ properties respectively. Let $\Delta'\subseteq \mathrm{K}_q(K)$ be the set of elements either of the form  $\beta|_K$ with $\beta \in \Delta_q$ or of the form $\{\alpha|_{K}, t^{1/p^r}\}$ for $\alpha \in \Delta_{q-1}$ and $r \in \mathbf{N}$. 
    Note that Lemma~\ref{lemma generators Kthry} ensures that the image of $\Delta'$ in $\mathrm{K}_q(K)/d$ generates for every integer $d$.\par
    
    Let $\beta \in \Delta_q$. We first check that the stably $C_i$ extension of $K$ realising $\beta|_K$ as a universal norm exists. By definition, there exists a stably $C_{i-1}$ algebraic extension $\tilde{k}/k$ such that $\beta$ is a universal norm of $\tilde{k}/k$. Let $\tilde{K}/K$ be the unramified extension lifting $\tilde{k}/k$. Note that $\tilde{K}$ is stably $C_i$ thanks to Proposition~\ref{prop stable greenberg general} and $\beta|_K$ is a universal norm of $\tilde{K}/K$ thanks to \cite[Lemma~7.3.6]{GilleSzamuely2017CentralGalois}.\par
    
    Let $\alpha \in \Delta_{q-1}$. Now, we construct a stably $C_i$ extension realising $\{\alpha|_K, t^{1/p^r}\}$ as a universal norm. By definition there exists a stably $C_i$ algebraic extension $\tilde{k}/k$ such that $\alpha$ is a universal norm of $\tilde{k}/k$. Let $K_0/K$ to be the unramified extension lifting $\tilde{k}/k$ and $\tilde{K}/K_0$ be a maximal totally ramified extension with respect to the $t$-adic valuation as in Lemma~\ref{lemma algebraically maximal extension t norm}. Since $\tilde{k}$ is stably $C_{i}$ and the value group of $\tilde{K}$ is divisible, the field $\tilde{K}$ is also stably $C_i$ thanks to Proposition~\ref{proposition transfer Puiseux}. Let $K'/K$ be a finite subextension of $\tilde{K}/K$ and $k'$ the algebraic closure of $k$ in $K'$. By our choice of $\tilde{K}$, we have $k' \subseteq \tilde{k}$ and there exists an element $t' \in K'$ such that $\mathrm{N}_{K'/k'K}(t') =  t^{1/p^r}$. Then there exists $\alpha' \in \mathrm{K}_{q-1}(k')$ such that $\mathrm{N}_{k'/k}(\alpha') = \alpha$. We deduce that $\{\alpha|_K , t^{1/p^r}\}$ belongs to $\mathrm{N}_q(K'/K)$ because
    \begin{align*}
        \mathrm{N}_{K'/K}(\{\alpha'|_{K'}, t'\}) &= \mathrm{N}_{k'K/K} \circ \mathrm{N}_{K'/k'K}(\{\alpha'|_{K'}, t'\}) \\
        &= \operatorname{N}_{k'K/K}(\{\alpha', t^{1/p^r}\})\\
        &= \{\alpha|_K, t^{1/p^r}\}
    \end{align*}

\end{proof}

\noindent The remainder of this section is devoted to various applications of this transfer result.
\subsection{Function fields}
As with the $C_i^q$ properties, we say that $k$ satisfies the (stably) uniform $C_i^q$ property away from the imperfection when the perfection of $k$ satisfies the (stably) uniform $C_i^q$~property. 
\begin{proposition}\label{prop KK function fields}
    Let $k_0$ be an algebraically closed field and $k$ a finite extension of $k_0(x_1,\cdots , x_n)$. Then, $k$ satisfies the uniform $C_i^q$ property for every pair of non-negative integers $i,q$ such that $i+q \geq n$. Moreover, $k$ satisfies the stably uniform $C_i^q$ property away from the imperfection for every pair $i,q$ such that $i+q\geq n$
\end{proposition}
\begin{proof}
    The fact that $k$ is stably $C_n$ is a direct consequence of Corollary~\ref{cor stable Lang-Nagata}. We can find a proof of the fact that it satisfies the uniform $C_i^q$ properties when $i+q \geq n$ in \cite[Theorem~2.2]{Diego2018KK}, but we repeat the argument for the convenience of the reader.\par
    Let $l/k$ be a finite extension of $k$. Let $u_1,\cdots u_q \in l^{\times}$ and $v_1,\cdots , v_r$ be a transcendence basis of the extension $l/k_0(u_{1}, \cdots , u_q)$ in $l$. Let $l'$ be the algebraic closure of $k_0(u_{1}, \cdots , u_q, v_{i+1}, \cdots , v_{r})$ in $l$. Note that the extension $l'/k_0$ has transcendence degree $q$ because 
        \[ \mathrm{trdeg}(k_0(u_{1}, \cdots , u_q)/k_0)=\mathrm{trdeg}(l/k_0)-\mathrm{trdeg}(k_0(u_1,...,u_q)/k_0)=m-r\]
        and $m-r + (r-i) = q$. Then, there exists a geometrically irreducible variety $X$ over $l'$ of dimension $m-q$ such that $l= l'(X)$. Since $l$ is $C_0^q$, the norm map $\mathrm{K}_{q}(l'') \to \mathrm{K}_{q}(l')$ is surjective for every finite extension $l''/ l'$ as reminded in the preliminaries. Then the symbol $\{u_1,\cdots , u_q\}$ is a norm for every finite subextension of $\tilde{l}:= \overline{l'}(X)$ and $\tilde{l}$ is a $C_{i}$ field thanks to \cite[Theorem~2a]{Nagata1957NotesPaperLang}. \par
        In order to establish the stably uniform $C_i^q$ away from the imperfection, one can repeat the previous argument. The only aditional remark is that we can assure that $\tilde{l}^{\mathrm{perf}}$ is stably $C_i$ thanks to Proposition~\ref{cor stable Lang-Nagata}
\end{proof}
\noindent We get the following corollary as a direct application of Theorem~\ref{thm general transfer}.
\begin{cor}
    Let $k_0$ be an algebraically closed field and $k$ be a $n$ dimensional function field over $k_0$. Then the field $k\lau{t_1}\dots\lau{t_m}$ satisfies the stably uniform $C_i^q$ property away from the imperfection for every pair of non-negative integers $i,q$ such that $i+q \geq n+m$
\end{cor}
\noindent Moreover, we can apply Lemma~\ref{lemma away from imperf implies property} to deduce some of Kato-Kuzumaki's properties \textit{at the imperfection}, that is without taking the perfection.
    \begin{cor}
        Let $i,q,n,m$ be non-negative integers such that $i+q = n+m$. Let $k_0$ be an algebraically closed field and $X$ be an integral variety over $k_0$ of dimension $\operatorname{min}\{q,m\}$. Then, 
        \begin{itemize}
            \item if $q\leq m$, the field $k_0(X)^{\mathrm{perf}}(x_{q+1} \dots, x_m)\lau{t_1}\dots \lau{t_n}$ satisfies $C_i^q$, and
            \item if $m\leq q$, the field $(k_0(X)\lau{t_1}\dots\lau{t_{q-m}})^{\mathrm{perf}}\lau{t_{q-m+1}}\dots \lau{t_n}$ satisfies $C_i^q$.
        \end{itemize} 
    \end{cor}
\subsection{Finite fields}
Another context where we can apply our general transfer result is the case of finite fields. Indeed, finite fields are stably $C_1$ thanks to \cite[Theorem~5.3.1]{Kartas2024geoC1}.
\begin{cor}
    The field $\mathbf{F}_p\lau{t_1}\dots\lau{t_m}$ satisfies $C_i^q$ away from the imperfection for every pair of non-negative integers $i,q$ such that $i+q \geq m+1$.
\end{cor}
\noindent Again, thanks to Lemma~\ref{lemma away from imperf implies property} we get the following corollary.
\begin{cor}\label{cor iterated laurent over finite}
    Let $j$ be in $\{1,\cdots , m\}$. Then the field
    \[
    \left(\mathbf{F}_p\lau{t_1}\dots\lau{t_{j-1}}\right)^{\mathrm{perf}}\lau{t_j}\dots\lau{t_m}
    \]
    satisfies $C_j^{m-j+1}$.
\end{cor}
\noindent Note that this corollary includes the $C_1^m$ for $\mathbf{F}_p\lau{t_1}\dots\lau{t_m}$ which was only known away from $p$, see \cite[Corollaire~4.7]{Wittenberg2015KKQp}.

\subsection{\texorpdfstring{$p$}{p}-adic fields}
The maximal unramified extension of a local field and any maximal totally ramified extension of a local field are $C_1$ thanks to  \cite[Theorem~12]{Lang1952QuasiAlgebraicClosure} and \cite[Corollary~1.1.3]{Kartas2024geoC1} respectively. Moreover, they are stably $C_1$ thanks to Proposition~\ref{proposition transfer Puiseux}. Then, Theorem~\ref{thm general transfer} admits the following corollary.
\begin{cor}\label{cor over p-adic fields}
    Let $k_0$ be a $p$-adic field and $k$ be either a maximal totally ramified extension of $k_0$ or the maximal unramified extension of $k_0$. Then the field $k\lau{t_1}\dots\lau{t_m}$ satisfies $C_i^q$ for every pair of non-negative integers $i,q$ such that $i+q \geq n$.
\end{cor}
\noindent This is enough to give an application for iterated Laurent series over $p$-adic fields.
\begin{proposition}\label{prop p-adic laurent C1m1}
    Let $k$ be a $p$-adic field. Then, $k$ satisfies the stably uniform $C_1^1$ property and the field $k\lau{t_1}\dots\lau{t_{m}}$ satisfies the stably uniform $C_1^{m+1}$ property.
\end{proposition}
\begin{proof}
    The arguments in the proof of \cite[Proposition~A.1]{HarryFelipe2025Stability} are enough to prove that satisfies the uniform $C_1^1$ property. We repeat the argument for the sake of the reader.\par
    Since a finite extension of a $p$-adic field is still a $p$-adic field, we might simplify the notation by not taking finite extensions. Let $\pi \in k^{\times}$ be a uniformiser. Let $k_{\pi}$ be the maximal abelian totally ramified extension of $k$ associated to $\pi$ via Lubin-Tate theory. Let $\tilde{k}$ be a maximal totally ramified extension of $k$ containing $k_{\pi}$. The field $K$ is stably $C_1$ thanks to \cite[Corollary~1.1.3]{Kartas2024geoC1} and Proposition~\ref{proposition transfer Puiseux}. Moreover, $\pi$ is a norm of every finite subextension of $\tilde{k}/k$ thanks to \cite[Theorem~3.5]{milneCFT} and \cite[Corollary~5.12]{Yoshida2008CFTLubinTate}. This is enough to conclude the stably uniform $C_1^1$ property because the set of uniformisers generate $k^{\times}$ as a group.\par
    For the second statement, we can proceed by induction on $m$. The case $m=0$ corresponds to the stably uniform $C_1^1$ property for $\mathbf{Q}_p$. Let $m> 0$ and assume that $k\lau{t_1}\dots\lau{t_{m-1}}$ satisfies $C_1^{m}$. Note that the field $k\lau{t_1}\dots\lau{t_{m-1}}$ satisfies the stably uniform $C_0^{m+1}$ property because its cohomological dimension is $m+1$. Then, by Theorem~\ref{thm general transfer}, we deduce that $k\lau{t_1}\dots\lau{t_{m}}$ satisfies the stably uniform $C_1^{m+1}$ property.

\end{proof}
\noindent We note that the $C_1^{m+1}$ property for $\mathbf{Q}_p\lau{t_1}\dots\lau{t_{m}}$ was previously known only \textit{away from} $p$  \cite[Corollaire~4.7]{Wittenberg2015KKQp}.

\section{The strong \texorpdfstring{$C_1^0$}{C10} property} \label{sec strong C1q}
In \cite{Wittenberg2015KKQp}, Wittenberg introduced the \textit{strong} $C_1^q$ which strengthens the $C_1^q$ property enough so that it admits transfer results along henselian valued fields, see \cite[Théorème~4.2]{Wittenberg2015KKQp}.
\begin{defi}\label{def strong C1q}
    Let $q$ be a non-negative integer. A field $k$ is said to satisfy the strong $C_1^q$ property if for every finite extension $k'/k$, proper $k'$-scheme $X$, and coherent sheaf $E$ over $X$ the group $\operatorname{K}_q(k')/\operatorname{N}_q(X/k')$ is a $\chi(X,E)$-torsion group.
\end{defi}
\noindent Given a field $k$ and a prime number $\ell$, we denote by $k(\ell)$ the maximal prime-to-$\ell$ extension of $k$. The proof of the following lemma is analogous to \cite[Lemma~2.5]{DiegoLuco2025TransferSerre2}, but we include it for completeness.
\begin{lemma}\label{lemma strong C_1^q reduces to l-special}
    Let $k$ be a field and $q$ a non-negative integer. The field $k$ satisfies the strong $C_1^q$ property if and only if form every prime number $\ell$ the field $k(\ell)$ satisfies the strong $C_1^q$ property.
\end{lemma}
\begin{proof}
    Assume that for every prime number $\ell$ the field $k(\ell)$ satisfies the strong $C_1^q$ property. Let $k'/k$ be a finite extension, $X$ a proper $k'$-scheme, and $E$ a coherent sheaf over $X$. Let $\alpha \in \operatorname{K}_q(k')$. Since $k(\ell)$ satisfies the strong $C_1^q$ property we can finite field extensions $K_1,\dots, K_n/k(\ell)$ such that $X(K_i) \neq \emptyset$ and 
    \[
    \chi(X,E) \cdot \alpha \in \left\langle \operatorname{N}_q(K_i/k(\ell) \, | \, i = 1,\dots ,n\right\rangle.
    \]
    We can also find a finite subextension $K/k'$ of $k'(\ell)/k'$ together with finite subextensions $K_i'/K$ of $K_i/K$ for every $i\in \{1,\dots , n\}$ such that $X(K'_i) \neq \emptyset$ and $\chi(X,E)\cdot\alpha$ belongs to $\langle\operatorname{N}_q(K_i'/K) \, | \, i \in \{1,\dots , n\}\rangle$. In particular, we have
    \[
    [K:k']\chi(X,E)\cdot \alpha \in \operatorname{N}_q(X/k').
    \]
    Since $\ell$ does not divide $[K:k']$ and this can be done for every prime number, a standard restriction-correstriction argument ensures that $\chi(X,E)\cdot \alpha \in \operatorname{N}_q(X/k')$. \par
    The converse is straight forward from the definition of the strong $C_1^q$ property.
\end{proof}

\noindent We also need a slight generalization of \cite[Theorem 1.2.2]{Kartas2024geoC1}:
\begin{proposition} \label{special fiber models}
Let $(K,v)$ be a tame field with divisible value group of rank $1$ and residue field $k$. Let $X$ be a $K$-variety and $d\in \N$. Suppose that for every $\Oo_K$-model $\Xx$ of $X$, there is a finite extension $k_\Xx/k$ of degree at most $d$ such that $\Xx(k_\Xx)\neq \emptyset$. Then there exists a finite extension $K_X/K$ of degree at most $d$ such that $X(K_X)\neq \emptyset$.
\end{proposition}
\begin{proof}
We outline the necessary modifications in the proof of \cite[Theorem 1.2.2]{Kartas2024geoC1}. Given a point $\xi \in X $, we write $\kappa(\xi)$ for the residue field at $\xi$. As in the proof of \cite[Lemma 4.2.1]{Kartas2024geoC1}, any sufficiently saturated elementary extension $k\preceq k^*$ has a finite extension $k'/k^*$ such that $\Xx(k')\neq \emptyset$ for each $\Oo_K$-model $\Xx$ of $X$, and such that these $k'$-points are compatible with respect to the maps of the inverse system of $\Oo_K$-models of $X$. Consider an elementary extension $(K,v)\preceq (K^*,v^*)$ with residue field $k^*$ and $v^*$ of rank $1$, and let $(K',v')$ be the unramified extension of $(K^*,v^*)$ lifting $k'/k^*$. 
By Raynaud's theory of formal models (see \cite[Theorem~2.22]{Scholze2012Perfectoid}), we have an isomorphism of locally ringed~spaces 
\begin{equation} \label{iso adic vs formal}
    (X^{\text{ad}},\Oo_{X^{\text{ad}}}^+)\cong \varprojlim_{\mathfrak{X}_i} (\mathfrak{X}_i,\Oo_{\mathfrak{X}_i}).
\end{equation} 
One can deduce from  \cite[\S~8.4 Lemma~2(c)]{Bosch2014LectureFormalRigid} and  \cite[\S~8.4 Lemma~6]{Bosch2014LectureFormalRigid} that algebraizable models are cofinal among all formal models. Then, the right hand side of \eqref{iso adic vs formal} has a $k'$-point and thus so does the left hand side.
This corresponds to a scheme-theoretic point $\xi \in X$ and a valuation $v_{\xi}$ on $\kappa(\xi)$ with residue field $k_{v_\xi}\hookrightarrow k'$. We then have $k'\models \text{Th}_{\exists, k}( k_{v_\xi})$ and also $\Gamma'\models \text{Th}_{\exists, \Gamma_0} (\Gamma_{v_\xi})$ since $\Gamma'$ is divisible. By Fact \ref{existential kuhlmann}, we get that $(K',v')\models \text{Th}_{\exists, K} (\kappa(\xi),v_{\xi})$. 
In particular, we get that $X(K')\neq \emptyset$. There is a sentence in the language of rings with parameters from $K$ expressing that $X$ admits a rational points in a finite extension of degree at most $d$. Since $K\preceq K^*$, there is a finite extension $K_X/K$ of degree at most $d$ such that $X(K_X)\neq \emptyset$. 
\end{proof}
\noindent The proof simplifies when $k$ is a finite field, which suffices for the applications here:
\begin{rem}
If $k=\mathbf{F}_q$, our assumption implies that $\Xx(\F_{q^d})\neq \emptyset$ for every $\Oo_K$-model of $\Xx$. By a compactness argument (in the sense of logic), we can even find such $\mathbf{F}_{q^d}$-rational points which are compatible with the maps of the inverse system.  
We thus get a scheme-theoretic point $\xi \in X$ and a valuation $v_{\xi}$ on $\kappa(\xi)$ with residue field $k_{v_\xi}=\mathbf{F}_{q^{f}}$ for some $f\mid d$. Let $K_X$ be the unramified extension of $K$ lifting $\mathbf{F}_{q^f}$. Kuhlmann's theorem now implies that $K_X\models \text{Th}_{\exists, K} (\kappa(\xi))$. In particular, we get that $X(K_X)\neq \emptyset$.
\end{rem}
\noindent We now extend \cite[Proposition 7.4]{Wittenberg2015KKQp} in arbitrary characteristic and for valuations of arbitrary rank. 
\begin{thm}\label{thm transfer along tame fields strong C10}
 Let $(K,v)$ be a tame field with divisible value group and residue field $k$. If $k$ satisfies the strong $C_1^0$ property, then so does $K$.
\end{thm}
\begin{proof}
Definition~\ref{def strong C1q} involves varieties defined over all finite extensions. However, since every finite extension of $K$ still satisfies the hypothesis, we might simplify the notation by just considering varieties over $K$. Let $X$ be a proper $K$-scheme. The Grothendieck group $G_0(X)$ is generated by the structure sheaves $\mathcal{O}_Z$ of closed subschemes $Z$ of $X$, see \cite[Section~8, Lemme~17]{BorelSerre1958TheoremeRiemannRoch}. Then, after replacing $X$ by $Z$ it is enough to prove that the index of every proper variety $X$ divides the Euler characteristic of its structure sheaf $\chi(X,\mathcal{O}_X)$. \par
We start with some reductions. Denote by $G_k$ and $G_K$ the absolute Galois groups of $k$ and $K$ respectively. \vspace*{1ex} \\
\textbf{Step 1:} Reduction to the case where $G_K$ is a pro-$\ell$ group for some prime number $\ell$. \vspace*{1ex} \\
Assume the theorem holds when $G_K$ is a pro-$\ell$ group. To deduce the general case, first note that $G_K\cong G_k$ by our assumptions on $(K,v)$. Therefore, for each prime number $\ell$, the field $K(\ell)$ is the unramified extension of $K$ obtained by lifting $k(\ell)/k$. Moreover, $K(\ell)$ is a tame field by \cite[Lemma 2.17(b)]{Kuhl} with divisible value group and residue field $k(\ell)$.  By Lemma \ref{lemma strong C_1^q reduces to l-special} and since $k$ is strongly $C_1^0$, the same is true for $k(\ell)$ for each prime $\ell$. By the pro-$\ell$ case, we get that $K(\ell)$ is strongly $C_1^0$ for each prime $\ell$. Again by Lemma \ref{lemma strong C_1^q reduces to l-special}, we conclude that $K$ is strongly $C_1^0$.\vspace*{1ex} \\
\textbf{Step 2:} Reduction to the case where $v$ has rank $1$. \vspace*{1ex}  \\
Assume the theorem holds for valued fields of rank \(1\). We first deduce it for finite rank valuations by induction on the rank.
Let \((K,v)\) be a valued field as in the statement of the theorem such that $v$ is of rank \(r\). The place associated to \(v\) factors as a composition of two places
\[
K \stackrel{w}\longrightarrow K_1 \stackrel{v_1}\longrightarrow k
\]
where \(w\) has rank \(1\) and \(v_1\) has rank \(r-1\). Moreover, both $(K,w)$ and $(K_1,v_1)$ are tame with divisible value group. By the induction hypothesis, $K_1$ is strongly $C_1^0$. By the rank $1$ case applied to $(K,w)$, it follows that $K$ is strongly $C_1^0$.

Now assume $v$ has arbitrary (possibly infinite) rank. Let $X$ be a $K$-variety. By Lemma \ref{tame spread out},
there exists a tame valued subfield $(K_0,v_0)\subseteq (K,v)$ such that there exists a proper $K_0$-scheme $X_0$ with $X_0 \otimes_{K_0} K \simeq X$ and $(K_0,v_0)$ has divisible value group of finite rank and residue field $k$. By the finite rank case, we know that $K_0$ satisfies the strong $C_1^0$ property. Since the index $\operatorname{I}(X)$ divides the index $\operatorname{I}(X_0)$ and $K_0$ satisfies the strong $C_1^0$ property, $\operatorname{I}(X)$ divides $\chi(X,\mathcal{O}_{X}) = \chi(X_0,\mathcal{O}_{X_0})$.\vspace*{1ex} \\
\textbf{Step 3:} Proof when $v$ is of rank $1$ and $G_k$ is a pro-$\ell$ group.\vspace*{1ex}  \\
Let $i,n$ be non-negative integers such that $\chi(X,\mathcal{O}_X)=\ell^i\cdot n$ where $\ell \nmid n$. 

Let $\Xx$ be an $\Oo_K$-model of $X$ and $\Xx_k=\Xx\otimes_{\Oo_K} k$ be the special fibre. We have that $\chi(\Xx_k, \mathcal{O}_{\Xx_k})=\chi(X,\mathcal{O}_X)=\ell^i\cdot n$ because $\Xx$ is flat over $\Oo_K$. Since $k$ is strongly $C_1^0$ and $G_k$ is a pro-$\ell$ group, there is a finite extension $k_\Xx$ of $k$ of degree at most $\ell^i$ such that $\Xx_k(k_\Xx)\neq \emptyset$. Proposition~\ref{special fiber models} ensures that there exists a finite extension $K'/K$ of degree at most $\ell^i$ such that $X(K')\neq \emptyset$. Since $G_K$ is a pro-$\ell$-group, we have that $[K':K]\mid \ell^i $ and hence $[K':K]\mid \chi(X,\mathcal{O}_X)$.
\end{proof}

\noindent Theorem~\ref{thm transfer along tame fields strong C10} allows us to establish the strong $C_1^1$ at $p$ for $\mathbf{Q}_p$ which is stated as an open question in \cite[Section~5]{Wittenberg2015KKQp}.
\begin{thm}
    The field $\mathbf{Q}_p$ satisfies the strong $C_1^1$ property.
\end{thm}
\begin{proof}
    Let $k/\mathbf{Q}_p$ be a finite extension, $X$ a proper $k$-variety, and $E$ a coherent sheaf over $X$. Let $\pi$ be a uniformiser of $k$ and $K/k$ a maximal totally ramified extension containing the maximal abelian totally ramified extension $k_{\pi}$ associated to $\pi$ through Lubin-Tate theory. Note that as in the proof of Proposition~\ref{prop p-adic laurent C1m1}, $\pi$ is a norm of any finite subextension of $K/k$.\par
    Since finite fields satisfy the strong $C_1^0$ fields (see \cite[Corollaire~3.5]{Wittenberg2015KKQp}), Theorem~\ref{thm transfer along tame fields strong C10} ensures that the $K$ satisfies the strong $C_1^0$ property. In particular, the index $\operatorname{I}(X_K)$ of $X_K := X\otimes_k K$ divides $\chi(X,E)$. Let $L_1,\dots, L_n/K$ be finite extensions such that $\operatorname{gcd}\{[L_j:K]\mid j=1,\dots ,n\} = \operatorname{I}(X_K)$. For every $i=1,\dots, n$ there exists a finite subextension $k_i/k$ of $K/k$ and a finite extension $l_i/k_i$ such that $[l_i:k_i] = [L_i:K]$ and $X(l_i) \neq \emptyset$. By our choice of $K$, there exists $\rho_i \in k_i$ such that $\operatorname{N}_{k_i/k}(\rho_i)= \pi$. Let $a_1,\dots,a_n \in \mathbf{Z}$ be such that $\sum_{i=1}^n a_i [l_i:k_i] = \operatorname{I}(X_K)$. Then,
    \begin{align*}
        \prod_{i=1}^n\operatorname{N}_{l_i/k}(\rho_i^{a_i})&=\prod_{i=1}^n\operatorname{N}_{k_i/k}(\rho_i)^{a_i[l_i:k_i]}\\
        &=\pi^{\operatorname{I}(X_K)}.
    \end{align*}
    This implies that $\pi^{\chi(X,E)}$ belongs to $\operatorname{N}_1(X/k)$ because $\operatorname{I}(X_K)$ divides $\chi(X,E)$. Since $k^{\times}$ is generated by the set of uniformisers, this proves that the group $k^{\times}/\operatorname{N}_1(X/k)$ is a $\chi(X,E)$-torsion group.
\end{proof}
\noindent Thanks to \cite[Théorème~4.2]{Wittenberg2015KKQp} we also get the following result which was only known away from $p$, see \cite[Corollaire~4.7]{Wittenberg2015KKQp}.
\begin{cor}
    The field $\mathbf{Q}_p\lau{t_2} \dots\lau{t_d}$ satisfies the strong $C_1^d$ property. 
\end{cor}

\section{Rationally connected varieties} \label{sec geometric variant}
    Koll\'ar defined the notion of a \textit{geometrically} $C_1$ field, replacing the hypersurfaces by smooth rationally connected varieties. Taking inspiration from this, we propose the following  geometric analogue of the $C_1^q$ property.
    \begin{defi}\label{defi geometric Ciq}
        Let $q$ be a non-negative integer. We say that $k$ satisfies $C_{rc}^q$ if for every finite extension $k'/k$ and smooth projective (geometrically) separably rationally connected variety $Z$ over $k'$ we have $\operatorname{K}_q(k') = \mathrm{N}_{q}(Z/k')$.
    \end{defi}
    \noindent Just as in the case of geometrically $C_1$ fields, \cite[Theorem 1.2]{HogadiXu2009Degenerations} ensures that $C_{rc}^q$ is stronger than $C_1^q$ in characteristic zero.
    \begin{proposition}
    Let $k$ be a field of characteristic $0$. If $k$ satisfies $C_{rc}^q$, it also satisfies $C_1^q$
    \end{proposition}
    \begin{proof}
        Let $k'/k$ be a finite extension and $Z \subseteq \mathbf{P}^n_{k'}$ be a hypersurface of degree $d$ with $d\leq n$. If $Z$ is smooth, it is a Fano variety. In particular, $Z$ is rationally connected thanks to \cite[Theorem V 2.13]{Kollar1995RationalCurvesVarieties} and we have $\operatorname{K}_q(k') =  \mathrm{N}_q(Z/k')$ by hypothesis. Even if $Z$ is not smooth, it can be expressed as a degeneration of a smooth Fano hypersurface. \cite[Theorem 1.2]{HogadiXu2009Degenerations} ensures that $Z$ has a geometrically integral rationally connected subvariety $X$. Since a resolution of singularities $X'$ of $X$ is a smooth rationally connected variety, we have $\operatorname{K}_q(k')= \mathrm{N}_q(X'/k')$. This ensures $\operatorname{K}_q(k') =  \mathrm{N}_q(X/k) \subseteq \mathrm{N}_q(Z/k)$.
    \end{proof}
    \noindent The main interest of introducing the $C_{rc}^q$ property, is that we have geometric tools at our disposal that do not apply directly to the $C_1^q$ property because it involves non-smooth hypersurfaces. Note that a field $k$ that satisfies the \textit{strong} $C_1^q$ property as defined in \cite[Definition 4.1]{Wittenberg2015KKQp} also satisfies $C_{rc}^q$ thanks to \cite[Corollary~3.8]{Kollar1995RationalCurvesVarieties} \par
    As a first example of a $C_{rc}^1$ we have any $p$-adic field, see \cite[Proposition~A.2]{HarryFelipe2025Stability}. We now prove that the $C_{rc}^q$ property admits a transfer principle in characteristic zero without using model theoretic tools.
    \begin{proposition}\label{prop transfer Puiseux}
        Let $q \in \mathbf{N}$ and $k$ a field of characteristic zero. Assume that $k$ satisfies $C_{rc}^q$. Then the field of Puiseux series $K = \bigcup_{n \geq 1} k(\!(t^{1/n})\!)$ satisfies $C^q_{rc}$.
    \end{proposition}
    \begin{proof}
        Note that every finite extension of $K$ is of the form $\bigcup_{n \geq 1} l(\!(t^{1/n})\!)$ for some finite extension $l/k$. Since $l$ also satisfies the hypothesis of the proposition, we may assume $k=l$. Then it is enough to check that for every rationally connected variety $Z$ over $K$ we have $\mathrm{K}_q(K) = \mathrm{N}_q(Z/K)$.\par
        Let $Z$ be a smooth projective rationally connected variety over $K$. For $n \in \mathbf{N}$ denote by $R_n$ the ring $R[\![t^{1/n}]\!]$ and $K_n$ its fraction field. We can apply \cite[Lemme 7.5]{Wittenberg2015KKQp} to find an $n \in \mathbf{N}$ and a regular, projective flat $R_n$-scheme $\mathcal{Z}$ whose special fibre $Y := \mathcal{Z} \otimes_{R_n} k$ is reduced and a strict normal crossing divisor of $\mathcal{Z}$ and such that $\mathcal{Z}\otimes_{R_n} K = Z$. We denote by $Z_n$ the rationally connected variety $\mathcal{Z} \otimes_{R_n} K_n$. \par
        Thanks to \cite[Theorem 1.2]{HogadiXu2009Degenerations} the $k$-variety $Y$ contains a geometrically integral rationally connected variety. Then $\mathrm{K}_q(k) = \mathrm{N}_q(Y/k)$ because $k$ satisfies $C^q_{rc}$. We begin by proving that $\mathrm{U}_q(K_n)$ is contained in $\mathrm{N}_q(Z_n/K_n)$. \par
    
    Let $\alpha \in \mathrm{U}_q(K_n)$ and $s:\mathrm{U}_q(K_n) \to \mathrm{K}_q(k)$ the natural specialisation map, see Section \ref{sec prelim and not}. There exists finite extensions $k_1, \cdots, k_n$ of $k$ and for each $i \in \{1,\cdots, n\}$ an element $\beta_i \in \mathrm{K}_q(k_i)$ such that $Y(k_i) \neq \emptyset$ for every $i \in \{1,\cdots, n\}$ and
        \[ s(\alpha) = \prod_{i=1}^n \mathrm{N}_{k_i/k}(\beta_i).\]
        Since the norm groups are a birational invariant of smooth varieties, see \cite[Remarque~5.7]{Wittenberg2015KKQp}, we may assume that $Y$ admits a $k_i$-point contained in only one irreducible component, i.e. $Y$ admits a smooth $k_i$-point for every $i \in \{1,\cdots, n\}$. Denote by $K_{n,i}$ the extension $k_i K_n$. Hensel's lemma ensures that $Z(K_{n,i}) \neq \emptyset$. We have the following commutative diagram with exact rows whose vertical arrow are the product of the norm maps
        \begin{equation} \label{diag norm and units in K-theory}
            \begin{tikzcd}
                0 \ar[r] & \prod_{i=1}^n\mathrm{U}_{q}^1(K_{n,i}) \ar[r] \ar[d] & \prod_{i=1}^n\mathrm{U}_{q}(K_{n,i}) \ar[r] \ar[d] & \prod_{i=1}^n\mathrm{K}_{q}(k_i) \ar[r] \ar[d] & 0 \\
                0 \ar[r] & \mathrm{U}_{q}^1(K_n) \ar[r]  & \mathrm{U}_{q}(K_n) \ar[r]  & \mathrm{K}_{q}(k) \ar[r] & 0.
            \end{tikzcd}
        \end{equation}
        Noting that $\mathrm{U}_{q}^1(K_n)$ is divisible, a diagram chase and a restriction-correstriction argument proves that $\alpha$ belongs to $\mathrm{N}_q(Z_n/K_n)$. \par
        Let $\alpha \in \mathrm{K}_q(K)$. We can find $m \in \mathbf{N}$ and $\alpha_0 \in \mathrm{K}_q(K_m)$ such that $\alpha  = \alpha_0|_K$. Moreover, we may assume that $m$ is big enough so that the previous paragraph ensures that $\mathrm{U}_q(K_m) \subseteq \mathrm{N}_q(Z_m/K_m)$. It is possible to write $\alpha_0$ as a product of symbols of the form $\{t^{1/m}, u_2, \cdots ,u_q \}$ and $\{u_1, u_2, \cdots ,u_q \}$  where $u_i$ is a unit in $R_m$ for every $i \in \{1,\cdots, q\}$ because such symbols generate $\mathrm{K}_q(K_m)$. We already know that symbols of the form $\{u_1, u_2, \cdots ,u_q \}$ are norms of $Z_m$. It would be enough to prove that symbols of the form $\{t^{1/m}, u_2, \cdots ,u_q \}$ become a norm after adding a big enough root of $t^{1/m}$. \par 
        Thanks to \cite[Théorème 7.5]{CT2011presqueRat} there exists a finite extension $k'/k$ such that $Z_m(k'(\!(t^{1/m})\!))$ is non-empty. Denote by $m'$ the degree of $k'/k$. Then $t^{1/m}$ belongs to $\operatorname{N}_1\left(k'(\!(t^{1/mm'})\!)/k(\!(t^{1/mm'})\!)\right)$. This means that $\{t^{1/m}, u_2, \cdots ,u_q \}|_{K_{mm'}}$ belongs to $\mathrm{N}_q(Z_{mm'}/K_{mm'})$, in particular, we have $\{t^{1/m}, u_2, \cdots ,u_q \}|_{K} \in \mathrm{N}_q(Z/K)$. This is enough to conclude that $\alpha$ belongs to $\mathrm{N}_q(Z/K)$.
    \end{proof}
    \begin{rmk}\label{rmk Puiseux transfer strong C1}
        A similar argument proves the same transfer principle for the strong $C_1^q$ property defined in \cite[Definition 4.1]{Wittenberg2015KKQp}.
    \end{rmk}
    \noindent This proposition is enough to give the following transfer principle.
    \begin{proposition}\label{prop transfer Laurent}
            Let $q \in \mathbf{N}$, and $k_0$ be a field of characteristic zero. If $k_0$ satisfies $C_{rc}^q$, the field of Laurent series $k:= k_0(\!(t)\!)$ satisfies $C_{rc}^{q+1}$.
        \end{proposition}
        \begin{proof}
            Let $X$ be a smooth projective rationally connected variety over $k$. Let $\pi \in k$ be an uniformiser and $u_1,\cdots ,u_q$ units in $k$. Denote by $K_{\pi}$ the field $\bigcup_{n\geq 1} k(\pi^{1/n})$. Proposition~\ref{prop transfer Puiseux} ensures that $K_{\pi}$ satisfies $C_{rc}^q$. In particular, we can find finite extensions $K_1,\cdots ,K_n$ of $K_{\pi}$ and for every $i \in \{1 ,\cdots, n\}$ an element $\alpha_i \in \mathrm{K}_q(K_i)$ such that $X(K_i) \neq \emptyset$ and 
            \[ \{u_1, \cdots , u_n\}|_{K_{\pi}} = \prod_{i=1}^n \mathrm{N}_{K_i/ K_{\pi}} (\alpha_i).\]
            This descends to a finite extension. Indeed, we can find  a finite subextension $\tilde{k}$ of $K_{\pi}/k$, finite extensions $k_1,\cdots k_n$ of $\tilde{k}$, and for $i \in \{1 , \cdots, n\}$ classes $\beta_i\in\mathrm{K}_q(k_i)$ such that
            \[\{u_1, \cdots , u_n\}|_{\tilde{k}} = \prod_{i=1}^n \mathrm{N}_{k_i/ \tilde{k}} (\beta_i).\]
            Since $\tilde{k}$ is contained in $K_{\pi}$, the uniformiser $\pi$ is a norm of the extension $\tilde{k}/k$. Fix $\lambda \in \tilde{k}^{\times}$ such that $\mathrm{N}_{\tilde{k}/k} (\lambda) = \pi$. Then,
            \begin{align*}
                \prod_{i=1}^n  \mathrm{N}_{k_i/ k} \{\lambda,\beta_i \}& = \prod_{i=1}^n \mathrm{N}_{\tilde{k}/k} \circ \mathrm{N}_{k_i/ \tilde{k}} \{\lambda,\beta_i \} \\
                & =  \mathrm{N}_{\tilde{k}/k}\{\lambda,\prod_{i=1}^n \mathrm{N}_{k_i/ \tilde{k}}(\beta_i) \} \\
                & =  \{\mathrm{N}_{\tilde{k}/k}(\lambda),\{u_1, \cdots , u_n\} \} \\
                & =  \{\pi,u_1, \cdots ,u_q \} 
            \end{align*}
            Since this is true for every uniformiser, we conclude that $\mathrm{N}_q(X/k) = \mathrm{K}_q(k)$.
        \end{proof}
        \noindent This gives a different proof of the $C_1^{m+1}$ property for $\mathbf{Q}_p(\!(t_1)\!) \cdots (\!(t_{m})\!)$ thanks to \cite[Proposition~A.2]{HarryFelipe2025Stability}.
        \begin{cor}
            Let $  d \geq 1$ be a natural number. The field $\mathbf{Q}_p(\!(t_1)\!) \cdots (\!(t_{m})\!)$ satisfies~$C_{rc}^{m+1}$.
        \end{cor}

\section*{Acknowledgments}
 The first-named author would like to thank Diego Izquierdo for his constant support and guidance. The authors also thank Olivier Wittenberg for his thoughtful comments. The second-named author acknowledges financial support from the Alexander von Humboldt Foundation and also thanks the École Normale Supérieure for its hospitality during a three-month visit.
    \bibliographystyle{alpha}
    \bibliography{ref.bib}
\end{document}